\documentclass[10pt]{amsart}
\usepackage{amscd,amssymb,hyperref}
\usepackage[PostScript=dvips]{diagrams}
\usepackage[dvips]{graphicx}

\usepackage{graphics}
\usepackage{epsfig}
\usepackage{picinpar}
\usepackage{wrapfig}
\usepackage{amsmath}
\usepackage{color}
\usepackage{pstricks,pst-node,pst-text,pst-3d}
\usepackage{pst-plot}
\usepackage{verbatim}

\newtheorem{thm}{Theorem}[section]
\newtheorem{lem}[thm]{Lemma}

\newtheorem{prop}[thm]{Proposition}

\newtheorem{rem}[thm]{Remark}
\newtheorem{example}[thm]{Example}

\def\square{\vbox{
      \hrule height 0.4pt
      \hbox{\vrule width 0.4pt height 5.5pt \kern 5.5pt \vrule width 0.4pt}
      \hrule height 0.4pt}}

\def\id{\mathrm{id}}

\def\Ker{\mathrm{K er}}
\def\ch\mathrm{c h}
\def\ab{\mathrm{a b}}

\newcommand{\Z}{\mathbb{Z}}

\newcommand{\calZ}{\ensuremath{\mathcal{Z}}}

\newcommand{\calG}{\ensuremath{\mathcal{G}}}
\newcommand{\calT}{\ensuremath{\mathcal{T}}}

\newcommand{\Index}{\mathrm{Index}}
\newcommand{\calB}{\mathcal{B}}

\newcommand{\AP}{\mathrm{AP}}

\let\la=\langle
\let\ra=\rangle

\numberwithin{equation}{section}

\begin{document}

\newcommand{\auths}[1]{\textrm{#1},}
\newcommand{\artTitle}[1]{\textsl{#1},}
\newcommand{\jTitle}[1]{\textrm{#1}}
\newcommand{\Vol}[1]{\textbf{#1}}
\newcommand{\Year}[1]{\textrm{(#1)}}
\newcommand{\Pages}[1]{\textrm{#1}}

\author{Roman Mikhailov}
\address{Steklov Mathematical Institute, Gubkina 8, 119991 Moscow, Russia}
\email{romanvm@mi.ras.ru}

\author{Jie Wu$^{\dag}$}
\address{Department of Mathematics, National University of Singapore, 2 Science Drive 2
Singapore 117542} \email{matwuj@nus.edu.sg}
\urladdr{www.math.nus.edu.sg/\~{}matwujie}

\thanks{$^{\dag}$ Research of the second author is supported in part by the AcRF Tier 1 (WBS No. R-146-000-137-112) and AcRF Tier 2 (WBS No. R-146-000-143-112) of MOE of Singapore and a grant (No. 11028104) of NSFC of China.}

%\thanks{The author would like to thank the hospitality of
%Peking University. Part of this project has been carried out during
%his visit to the Institute of Mathematics of Peking University from
%June 19 to July 3, 2006.}

\title{A combinatorial description of homotopy groups of spheres$^{*}$}
\thanks{$^{*}$ Research for this article was partially supported by a grant (No.11028104) of NSFC of China.}
\begin{abstract}
 We give a combinatorial description of general homotopy groups of $k$-dimensional spheres with $k\geq3$ as well as those of Moore spaces.
 For $n>k\geq 3,$ we construct a finitely generated group defined by explicit generators and relations, whose center is exactly $\pi_n(S^k)$.
\end{abstract}

\maketitle

\section{Introduction}
The purpose of this article is to give an explicit combinatorial
description of general homotopy groups of $k$-dimensional spheres
with $k\geq3$ as well as those of Moore spaces. The description is
given by identifying the homotopy groups as the center of a
quotient group of the self free products with amalgamation of pure
braid groups by certain symmetric commutator subgroups.

A combinatorial description of $\pi_*(S^2)$ was discovered by the
second author in 1994 and given in his thesis~\cite{Wu3}, with a
published version in~\cite{Wu2}. This description can be briefly
summarized as follows. Let $F_n$ be a free group of rank $n\geq 1$
with a basis given by $\{x_1,\ldots,x_n\}$. Let $R_i=\la
x_i\ra^{F_n}$ be the normal closure of $x_i$ in $F_n$ for $1\leq
i\leq n$. Let $R_{n+1}=\la x_1x_2\cdots x_n\ra^{F_n}$ be the
normal closure of the product element $x_1x_2\cdots x_n$ in $F_n$.
We can form a symmetric commutator subgroup
$$
[R_1,R_2,\ldots,
R_{n+1}]_S=\prod_{\sigma\in\Sigma_{n+1}}[\dots[R_{\sigma(1)},R_{\sigma(2)}],\ldots,R_{\sigma(n+1)}].
$$
This gives an explicit subgroup of $F_n$ with a set of generators that can be understood by taking a collection of iterated commutators. By~\cite[Theorem 1.4]{Wu2}, we have the following combinatorial description on $\pi_*(S^2)$.

\begin{thm} For $n\geq 1,$ there is an isomorphism
$$
\pi_{n+1}(S^2)\simeq \frac{R_1\cap \dots \cap R_{n+1}}{[R_1,\dots,
R_{n+1}]_S}
$$
Moreover, the homotopy group $\pi_{n+1}(S^2)$ is isomorphic to the
center of the group $F_n/[R_1,R_2,\ldots, R_{n+1}]_S$.\hfill
$\Box$
\end{thm}

The groups $F_n/[R_1,R_2,\ldots, R_{n+1}]_S$ can be defined using
explicit generators and relations. This situation is very
interesting from the group-theoretical point of view: we don't
know how to describe homotopy groups $\pi_*(S^2)$ in terms of
generators and relations, but we can describe a bigger group whose
center is exactly $\pi_*(S^2)$.

It has been the concern of many people whether one can give a
combinatorial description of homotopy groups of higher dimensional
spheres, ever since the above result was announced in 1994.
Technically the proof of this theorem was obtained by determining
the Moore boundaries of Milnor's $F[K]$-construction~\cite{Milnor}
on the simplicial $1$-sphere $S^1$, which is a simplicial group
model for $\Omega S^2$. A canonical approach is to study Milnor's
construction $F[S^k]\simeq \Omega S^{k+1}$ for $k>1$. Although
there have been some attempts~\cite{ZW} to study this question
using $F[S^k]$, technical difficulties arise in handling Moore
boundaries of $F[S^k]$ in a good way, and  combinatorial
descriptions of homotopy groups of higher dimensional spheres
using simplicial group model $F[S^k]$ would be very messy.

In this article, we give a combinatorial description of
$\pi_*(S^k)$ for any $k\geq3$ by using free product with
amalgamation of pure braid groups. Our construction is as follows.
Given $k\geq 3,\ n\geq 2$, let $P_n$ be the $n$-strand Artin pure
braid group with the standard generators $A_{i,j}$ for $1\leq
i<j\leq n$. We construct a subgroup $Q_{n,k}$ of $P_n$ from
cabling as follows. Our cabling process starts from $P_2=\Z$
generated by the $2$-strand pure braid $A_{1,2}$.
\begin{enumerate}
\item[]\textbf{Step 1.} Consider the $2$-strand pure braid $A_{1,2}$. Let $x_i$ be $(k-1)$-strand braid obtained by inserting $i$ parallel strands into the tubular neighborhood of the first strand of $A_{1,2}$ and $k-i-1$ parallel strands into the tubular neighborhood of the second strand of $A_{1,2}$ for $1\leq i\leq k-2$. The picture of $x_i$ is as follows:
\begin{center}
\epsfig{file=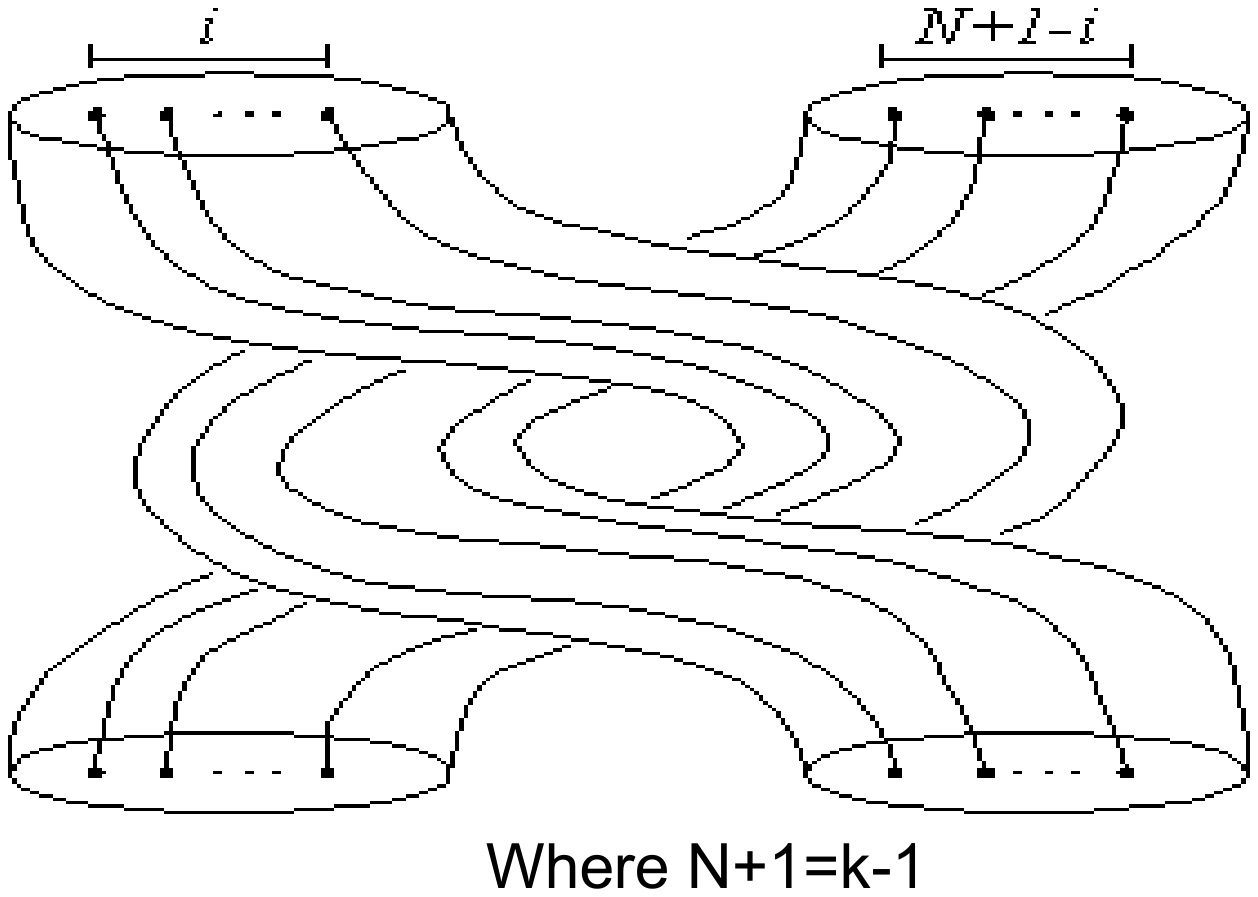, bbllx=0, bblly=0, bburx=488, bbury=72,
width=5in, height=3in, clip=}
\end{center}
\item[]\textbf{Step 2.} Let $\alpha_{k}=[\ldots [[x_1^{-1},
x_1x_2^{-1}],x_2x_3^{-1}],\ldots, x_{k-3}x_{k-2}^{-1},x_{k-2}]$ be
the $(k-1)$-strand braid. \item[]\textbf{Step 3.} By applying the
cabling process as in Step 1 to the element $\alpha_{k}$, we
obtain the $n$-strand braids $y_j$ for $1\leq j\leq
\binom{n-1}{k-2}$.
\end{enumerate}
Let $Q_{n,k}$ be the subgroup of $P_n$ generated by $y_j$ for $1\leq j\leq \binom{n-1}{k-2}$.
Now consider the free product with amalgamation
$$
P_n\ast_{Q_{n,k}} P_n.
$$
Let $A_{i,j}$ be the generators for the first copy of $P_n$ and let $A'_{i,j}$ denote the generators $A_{i,j}$ for the second copy of $P_n$. Let $R_{i,j}=\la A_{i,j},A'_{i,j}\ra^{P_n\ast_{Q_{n,k}}P_n}$ be the normal closure of $A_{i,j}, A'_{i,j}$ in $P_n\ast_{Q_{n,k}}P_n$. Let
$$
[R_{i,j}\ | \ 1\leq i<j\leq n]_S=\prod_{\{1,2,\ldots,n\}=\{i_1,j_1,\ldots,i_t,j_t\}}[[R_{i_1,j_1},R_{i_2,j_2}],\ldots,R_{i_t,j_t}]
$$
be the product of all commutator subgroups such that each integer $1\leq j\leq n$ appears as one of indices at least once. By Lemma~\ref{lemma3.6}, this product can be given by taking over those commutator subgroups $[[R_{i_1,j_1},R_{i_2,j_2}],\ldots,R_{i_t,j_t}]$ such that
\begin{enumerate}
\item[1)] $\{i_1,j_1,\ldots,i_t,j_t\}=\{1,2,\ldots,n\}$ and
\item[2)]
$\{i_1,j_1,\ldots,i_t,j_t\}\smallsetminus\{i_p,j_p\}\not=\{1,2,\ldots,n\}.$
\end{enumerate}
Our main theorem is as follows:

\begin{thm}\label{theorem1.2}
Let $k\geq3$. The homotopy group $\pi_n(S^k)$ is isomorphic to the
center of the group
$$
(P_n\ast_{Q_{n,k}}P_n)/[R_{i,j}\ | \ 1\leq i<j\leq n]_S
$$
for any $n$ if $k>3$ and any $n\not=3$ if $k=3$.
\end{thm}

\vspace{.25cm} \noindent\textbf{Note.} The only exceptional case
is that $k=3$ and $n=3$. In this case, $\pi_3(S^3)=\Z$ while the
center of the group is bigger than $\Z$.

The center of the group $(P_n\ast_{Q_{n,k}}P_n)/[R_{i,j}\ | \
1\leq i<j\leq n]_S$ is in fact given by Brunnian-type braids in
the following sense: Let $\bar d_k\colon P_n\to P_{n-1}$ be the
operation of removing the $k$-th strand for $1\leq k\leq n$. A
Brunnian braid means an $n$-braid $\beta$ such that $\bar d_k
\beta=1$ for any $1\leq k\leq n$. Namely $\beta$ becomes a trivial
braid after removing any one of its strands. This notion can
canonically be extended to free products of braid groups. In other
words, we have a canonical operation $\bar d_k\colon P_n\ast
P_n\to P_{n-1}\ast P_{n-1}$ which is a group homomorphism such
that, for each $n$-braid $\beta$ in the first copy of $P_n$ or the
second copy of $P_n$, $\bar d_k\beta$ is the $(n-1)$-strand braid
given by removing the $k$-th strand of $\beta$. A Brunnian-type
word means a word $w$ such that $\bar d_k w=1$ for any $1\leq
k\leq n$. Without taking amalgamation, it can be seen from our
techiniques that the Brunnian-type braids are exactly given by the
symmetric commutator subgroup $[R_{i,j}\ | \ 1\leq i<j\leq n]_S$.
However the question on determining Brunnian-type braids after
taking amalgamation becomes very tricky. The question here is
about the self free product of $P_n$ with the amalgamation given
by the subgroup $Q_{n,k}$. It is straightforward to check that the
strand-removing operation $\bar d_k$ maps $Q_{n,k}$ into
$Q_{n-1,k}$ and so the removing operation $\bar d_k\colon
P_n\ast_{Q_{n,k}}P_n\to P_{n-1}\ast_{Q_{n-1,k}}P_{n-1}$ is a
well-defined group homomorphism. From our construction of
simplicial groups given by free products with amalgamation, the
Brunnian-type braids in $P_n\ast_{Q_{n,k}}P_n$  are exactly the
Moore cycles in our simplicial group model for $\Omega S^k$ and so
the center\footnote{For a group $G$, we denote its center by
$Z(G)$.}
$$Z((P_n\ast_{Q_{n,k}}P_n)/[R_{i,j}\ | \ 1\leq i<j\leq n]_S)\cong
\pi_n(S^k)$$ is exactly given by the Brunnian-type braids in
$P_n\ast_{Q_{n,k}}P_n$ modulo the subgroup $[R_{i,j}\ | \ 1\leq
i<j\leq n]_S$. One important point concerning Brunnian-type braids
of the self free product with amalgamation of $P_n$ is that the
homotopy groups $\pi_n(S^k)$ can be given as quotient groups for
any $k\geq 3$.

Mark Mahowald asked in 1995 whether one can give a combinatorial
description of the homotopy groups of the suspensions of real
projective spaces. In this article, we also give a combinatorial
description of the homotopy groups of Moore spaces as the first
step for attacking Mahowald's question. Let $M(\Z/q,k)$ be the
$(k+1)$-dimensional Moore space. Namely
$M(\Z/q,k)=S^k\cup_qe^{k+1}$ is the homotopy cofibre of the degree
$q$ map $S^k\to S^k$. If $k\geq3$, we give a combinatorial
description of $\pi_*(M(\Z/q,k))$ given as the centers of quotient
groups of threefold self free product with amalgamation of pure
braid groups, which is similar to the description given in
Theorem~\ref{theorem1.2}. (The detailed description will be given
in Section~\ref{section4}.) This description is less explicit then
the one given in Theorem~\ref{theorem1.2}, but it leads to
combinatorial descriptions of homotopy groups of finite complexes
from iterated self free products with amalgamations of pure braid
groups.

For the homotopy groups of $3$-dimensional Moore spaces, there is
an explicit combinatorial description that deserves to be
described here as it arises in certain divisibility questions
concerning braids. Let $x_1,\ldots,x_{n-1}$ be $n$-strand braid
obtained by cabling $A_{1,2}$ as described in step 1 of the
construction for the group $Q_{n,k}$. It was proved in~\cite{CW1}
that the subgroup of $P_n$ generated by $x_1,\ldots,x_{n-1}$ is a
free group of rank $n-1$ with a basis given by
$x_1,\ldots,x_{n-1}$. Let $F_{n-1}=\la x_1,\ldots,x_n\ra\leq P_n$
be the subgroup generated by $x_1,\ldots,x_{n-1}$. Given an
integer $q$, since $F_{n-1}=\la x_1,\ldots,x_{n-1}\ra$ is free,
there is a group homomorphism $\phi_q\colon F_{n-1}\to F_{n-1}$
such that $\phi_q(x_j)=x_j^q$ for $1\leq j\leq n-1$. Now we form a
free product with amalgamation by the push-out diagram
\begin{diagram}
F_{n-1}&\rInto& P_n\\
\dInto>{\phi_q}&&\dTo\\
F_{n-1}&\rTo&P_n\ast_{\phi_q}F_{n-1},\\
\end{diagram}
namely the group $P_n\ast_{\phi_q}F_{n-1},$ which is the free
product by identifying the subgroup $F_{n-1}$ with the subgroup of
$F_{n-1}$ generated by $x_1^q,\ldots,x_{n-1}^q$ in a canonical
way. Let $y_j$ denote the generator $x_j$ for $F_{n-1}$ as the
second factor in the free product $P_n\ast_{\phi_q}F_{n-1}$ for
$1\leq j\leq n-1$. Let
$$
R_1=\la y_1\ra^{P_n\ast_{\phi_q}F_{n-1}}, R_j=\la
y_{j-1}y^{-1}_j\ra^{P_n\ast_{\phi_q}F_{n-1}}, R_{n}=\la
y_{n-1}\ra^{P_n\ast_{\phi_q}F_{n-1}}
$$
be the normal closure of $y_1, y_{j-1}y_j^{-1},y_{n-1}$ in
$P_n\ast_{\phi_q}F_{n-1}$, respectively, for $2\leq j\leq n-1$.
Let $R_{s,t}=\la A_{s,t}\ra^{P_n\ast_{\phi_q}F_{n-1}}$ be the
normal closure of $A_{s,t}$ in $P_n\ast_{\phi_q}F_{n-1}$ for
$1\leq s<t\leq n$. Define the index set $\Index(R_j)=\{j\}$ for
$1\leq j\leq n$ and $\Index(R_{s,t})=\{s,t\}$ for $1\leq s<t\leq
n$. Now define the symmetric commutator subgroup
$$
[R_i, R_{s,t}\ | \ 1\leq i\leq n,1\leq s<t\leq n]_S=\prod_{\{1,2,\ldots,n\}=\bigcup\limits_{j=1}^t\Index(C_j)}[[C_1,C_2],\ldots,C_t],
$$
where each $C_j=R_i$ or $R_{s,t}$ for some $i$ or $(s,t)$.

\begin{thm}\label{theorem1.3}
The homotopy group $\pi_n(M(\Z/q,2))$ is isomorphic to the center of the group
$$
(P_n\ast_{\phi_q}F_{n-1})/[R_i, R_{s,t}\ | \ 1\leq i\leq n,1\leq s<t\leq n]_S
$$
for $n\not=3$.
\end{thm}

\noindent\textbf{Note.} For the exceptional case $n=3$,
$\pi_3(M(\Z/q,2))$ is contained in the center but the equality
fails.

Some remarks concerning the methodology of this paper are given
next. The notion of simplicial sets and simplicial groups have
been largely studied since it was introduced in the early of
1950s, when D. Kan established the foundational work for
simplicial homotopy theory~\cite{Kan, Kan2}. Various important
results have been achieved by studying simplicial groups. For
instance, the Adams spectral sequence can be obtained from the
lower central series of Kan's construction~\cite{BCKQRS} for
computational purpose on homotopy groups. A combinatorial
description of general homotopy groups of $S^2$ was given
in~\cite{Wu2} with important progress in connecting to Brunnian
braids~\cite{BCWW}. This description was generalized in~\cite{EM}
by studying van Kampen-type theorem for higher homotopy groups.
Serious study of Brunnian braids~\cite{BMVW,LW2} introduced the
notion of symmetric commutator subgroups in determining the group
of Brunnian braids on surfaces $S$ for $S\not=S^2$ or
$\mathbb{R}\mathrm{P}^2$. By using this notion together with the
embedding theorem in~\cite[Theorem 1.2]{CW1} as well as the
Whitehead Theorem on free products with amalgamation of simplicial
groups~\cite[Proposition 4.3]{Kan-Thurston}, we are able to
control the Moore boundaries of our simplicial group models for
the loop spaces of spheres and Moore spaces, which leads to our
results.

Theorems~\ref{theorem1.2} and~\ref{theorem1.3} have more theoretical significance rather than computational purpose. It addresses the importance and complexity on the questions concerning Brunnian-type braids in free products with amalgamation of braid groups.

The article is organized as follows. In Section 2, we study free
products with amalgamation of simplicial groups. In some cases,
these products present simplicial models for loop spaces of
homotopy push-out spaces. In Section 3, for $k\geq 3,$ we
construct simplicial groups $\calT(S^k;\alpha)$ such that there is
a homotopy equivalence $|\calT(S^k;\alpha)|\simeq \Omega S^k.$
There is a natural way to describe Moore boundaries of
$\calT(S^k;\alpha)$ and this description is a key point in the
proof of Theorem \ref{theorem1.2} which we give in Section 3. In
Section 4, we consider triple free products with amalgamation of
simplicial braid groups and construct simplicial models for loop
spaces for Moore spaces. For $k\geq 3,$ we give a description of a
finitely-generated group such that its center is
$\pi_n(M(\Z/q,k))$ (Theorem \ref{theorem4.4}). Section 5 is about
3-dimensional Moore spaces. In this case, the simplicial models
for loop spaces of Moore spaces can be simplified. We prove
Theorem \ref{theorem1.3} in Section 5.

This article was finished during the visit of both authors to
Dalian University of Technology under the support of a grant
(No.11028104) of NSFC of China in July of 2011. The authors would
like to thank the hospitality of Dalian University of Technology
for supporting our research on this topic.

\vspace{.5cm}
\section{Free Products with Amalgamation on Simplicial Groups}
\vspace{.5cm}

Let $\phi\colon G\to G'$ and $\psi\colon G\to G''$ be group
monomorphisms. Then we have the free product with amalgamation
$G'\ast_GG''$. More precisely $G'\ast_GG''$ is the quotient group
of the free product $G'\ast G''$ by the normal closure of the
elements $\phi(g)\psi(g)^{-1}$ for $g\in G$. The group
$G'\ast_GG''$ has the universal property that the following
diagram
\begin{diagram}
G&\rInto^{\phi}&G'\\
\dInto>{\psi}&&\dTo\\
G''&\rTo&G'\ast_GG''\\
\end{diagram}
is a pushout diagram in the category of groups. Let $G'=\la X' \ |
\ R'\ra$ and $G''=\la X'' \ | \ R''\ra$ be presentations of the
groups $G'$ and $G''$, respectively. Let $X$ be a set of
generators for the group $G$. Then the group $G'\ast_GG''$ has a
presentation
$$
G'\ast_GG''=\la X', X'' \ | \ R', \ R'', \ \phi(x)\psi(x)^{-1} \textrm{ for } x\in X\ra.
$$
In particular, if $X'$, $X''$, $R'$, $R''$ and $X$ are finite sets, then $G'\ast_GG''$ is a finitely presented group with a presentation given as above. The notion of free product with amalgamation can be canonically extended to the category of simplicial groups.

Recall that a simplicial group $G$ consists in a sequence of
groups $G=\{G_n\}_{n\geq 0}$ with face homomorphisms $d_i\colon
G_n\to G_{n-1}$ and degeneracy homomorphisms $s_i\colon G_n\to
G_{n+1}$ for $0\leq i\leq n$ such that the following simplicial
identities holds:
\begin{enumerate}
\item[1)] $\Delta$-identity: $d_id_j=d_jd_{i+1}$ for $i\geq j$,
\item[2)] Degeneracy Identity: $s_is_j=s_{j+1}s_i$ for $i\leq j$,
\item[3)] Mixing Relation:
$$
d_is_j=\left\{
\begin{array}{rcl}
s_{j-1}d_i&\textrm{ if }& i<j,\\
\id&\textrm{ if }& i=j,j+1,\\
s_jd_{i-1}&\textrm{ if }& i>j+1.\\
\end{array}\right.
$$
\end{enumerate}
A simplicial homomorphism $f\colon G\to G'$ consists in a sequence
of group homomorphism $f=\{f_n\}$ with $f_n\colon G_n\to G'_n$
such that $d^{G'}_if_n=f_{n-1}d^G_i$ and
$s^{G'}_if_n=f_{n+1}s^{G}_i$ for $0\leq i\leq n$. A simplicial
monomorphism $f\colon G\to G'$ means a simplicial homomorphism
$f=\{f_n\}$ such that each $f_n\colon G_n\to G'_n$ is a
monomorphism. Similarly we have the notion of simplicial
epimorphism.

For a simplicial group $G$, recall that the Moore chain complex $N_*G$ is defined by
$$
N_nG=\bigcap_{j=1}^n\Ker(d_i\colon G_n\to G_{n-1})
$$
with the differential given by the restriction of the first face $d_0|\colon N_nG\to N_{n-1}G$. The Moore chain complex functor has the following important properties. For a simplicial set $X$, let $|X|$ denote its geometric realization.

\begin{prop}\label{proposition2.1}
The following statements hold:
\begin{enumerate}
\item[1)] Let $G$ be any simplicial group. Then there is a natural isomorphism
$$H_n(N_*G;d_0|)\cong \pi_n(|G|)$$ for all $n$.
\item[2)] Let $f\colon G\to G'$ be a simplicial homomorphism. Then $f$ is a simplicial monomorphism (epimorphism) if and only if
$$
N(f)\colon N_qG\longrightarrow N_qG'
$$
is a monomorphism (epimorphism) for all $q$.
\item[3)] A sequence of simplicial groups
$$1\to G'\to G\to G''\to 1$$
is short exact if and only if the corresponding sequence of Moore chain complexes
$$
1\to N_*G'\to N_*G\to N_*G''\to 1
$$
is short exact.
\end{enumerate}
\end{prop}
\begin{proof}
Assertion (1) is the classical theorem of John Moore, see the
survey paper~\cite{Curtis}. Assertion (2) is given in Quillen's
book~\cite[Lemma 5, 3.8]{Quillen}.

(3). By~\cite[Proposition 4.1.4]{BCWW}, the Moore chain functor is
an exact functor. We show that the inverse statement is also true.
Namely if $1\to N_*G'\to N_*G\to N_*G''\to 1$ is short exact, then
$1\to G'\to G\to G''\to 1$ is short exact. By assertion (2),
$G'\to G$ is a simplicial monomorphism and $G\to G''$ is a
simplicial epimorphism. From Conduch\'e's decomposition theorem of
simplicial groups~\cite{Cond}, the composite $G'\to G\to G''$ is
trivial and so $G'$ is mapped into $\Ker(G\to G'')$. Since
$N_*(G')\cong N_*\Ker(G\to G'')=\Ker(N_*G\to N_*G'')$, $G'\to
\Ker(G\to G'')$ is an isomorphism by assertion (2) and the result
follows.
\end{proof}

Let $\calZ_nG=\bigcap_{j=0}^n\Ker(d_i\colon G_n\to G_{n-1})\leq
N_nG$ be the \textit{Moore cycles} and let
$\calB_n\calG=d_0(N_{n+1}G)\leq \calZ_nG$ be the \textit{Moore
boundaries}. By assertion (1), the homotopy group $\pi_n(|G|)$ is
given by $\calZ_nG/\calB_nG$.

The construction of free product with amalgamation on simplicial
groups is given in the same way. Let $\phi\colon G\to G'$ and
$\psi\colon G\to G''$ be simplicial monomorphisms. Then
$G'\ast_GG''$ is a simplicial group with each $(G'\ast_GG'')_n$ is
the free product with amalgamation of $G'_n\ast_{G_n}G''_n$ for
the group homomorphisms $\phi_n\colon G_n\to G'_n$ and
$\psi_n\colon G_n\to G''_n$. The face homomorphisms are (uniquely)
determined by the pushout property:
\begin{diagram}
G_n          &              &\rInto^{\phi_n} &                     &G'_n &  &\\
            &\rdInto>{\psi_n}&              &                      &{d_i}^{G'_i}\dTo&\rdTo&\\
\dTo>{d_i^G}&              &G''_n         &\rTo   &              &                &G'_n\ast_{G_n}G''_n\\
 G_{n-1}    & \rInto^{\phi_{n-1}}& &     &G'_{n-1}&  &\\
           & \rdInto>{\psi_{n-1}}&\dTo>{d_i^{G''}}&   &        &\rdTo& \dDashto>{d_i^{G'\ast_GG''}}\\
          &                  &G''_{n-1}       &\rTo &   &&G'_{n-1}\ast_{G_{n-1}}G''_{n-1}.\\
\end{diagram}
Similarly the degeneracy homomorphisms are (uniquely) determined
by the pushout property. The uniqueness of the induced face and
degeneracy homomorphisms forces the simplicial identities to hold
for $d^{G'\ast_GG''}_i$ and $s^{G'\ast_GG''}_j$ and so
$G'\ast_GG''$ becomes a simplicial group. If we write the elements
$w$ in $(G'\ast_GG'')_n=G'_n\ast_{G_n}G''_n$ in terms of words as
a product of elements from $G'_n$ or $G''_n$, then
$d^{G'\ast_GG''}_i(w)$ is given by applying $d_i^{G'}$ or
$d_i^{G''}$ to the factors of $w$. Similarly we can compute
degeneracy homomorphism $s_i^{G'\ast_GG''}$ on $(G'\ast_GG'')_n$
in the same manner.

There is a classifying space functor from the category of
simplicial groups to the category of simplicial sets, denoted by
$\bar W$, with the property that the geometric realization of
$\bar W(G)$ is a classifying space of the geometric realization of
the simplicial group $G$. We refer to Curtis' paper~\cite{Curtis}
for the detailed construction of the functor $\bar W$.

An important property of free product with amalgamation on
simplicial groups is that the classifying space of $G'\ast_GG''$
can be controlled. This property is a simplicial consequence of
the classical asphericity result of J. H. C.
Whitehead~\cite[Theorem 5]{Whitehead} in 1939 and the formal
statement of the following theorem was given in Kan-Thurston's
paper~\cite[Proposition 4.3]{Kan-Thurston}.

\begin{thm}[Whitehead Theorem]\label{theorem2.1}
Let $\phi\colon G\to G'$ and $\psi\colon G\to G''$ be simplicial monomorphisms. Then the classifying space $\bar W(G'\ast_GG'')$ is the homotopy push-out of the diagram
\begin{diagram}
\bar W G&\rTo^{\bar W\phi}& \bar WG'\\
{\bar W\psi}\dTo&\mathrm{push}&\dTo\\
\bar W G''&\rTo&\bar W (G'\ast_GG'').\\
\end{diagram}\hfill $\Box$
\end{thm}

\vspace{.5cm}
\section{Description of Homotopy Groups of Spheres and Proof of Theorem~\ref{theorem1.2}}
\vspace{.5cm}

 In this section, we are going to construct a
simplicial group model $\calT(S^k)$ for $\Omega S^{k}$, $k\geq 3$,
by using pure braid groups. From this, we are able to give a
combinatorial description of the homotopy group $\pi_q(S^k)$ for
general $q$.

\subsection{Milnor's $F[K]$-construction on spheres}
Let $K$ be a simplicial set with a fixed choice of base-point
$s_0^nx_0\in K_n$. Milnor~\cite{Milnor} constructed a simplicial
group $F[K]$ where $F[K_n]$ is the free group generated by $K_n$
subject to the single relation that $s_0^nx_0=1$. The face and
degeneracy homomorphisms on $F[K]$ are induced by the face and
degeneracy functions on $K$. An important property of Milnor's
construction is that the geometric realization $|K|$ of $F[K]$ is
homotopy equivalent to $\Omega\Sigma |K|$. (Note. In Milnor's
paper~\cite{Milnor}, $K$ is required to be a reduced simplicial
set. This result actually holds for any pointed simplicial set by
a more general result~\cite[Theorem 4.9]{Wu1}.)

We are interested in specific simplicial group models for $\Omega
S^{k+1}$ and so we start by considering the simplicial $k$-sphere
$S^k$. Recall that the simplicial $k$-simplex $\Delta[k]$ can be
defined explicitly as follows:
\begin{enumerate}
\item[] \textit{$\Delta[k]_n=\{(i_0,i_1,\ldots,i_n)\ | \ 0\leq i_0\leq i_1\leq \cdots \leq i_n\leq k\}$ with $d_i\colon \Delta[k]_n\to \Delta[k]_{n-1}$ given by removing the $(i+1)$st coordinate and $s_i\colon \Delta[k]_n\to \Delta[k]_{n+1}$ given by doubling the $(i+1)$st coordinate for $0\leq i\leq n$.}
\end{enumerate}
Let $\sigma_k=(0,1,\ldots,k)\in \Delta[k]_k$ and let $\partial\Delta[k]$ be the simplicial subset of $\Delta[k]$ generated by the faces $d_0\sigma_k,\ldots,d_k\sigma_k$. Namely $\partial\Delta[k]$ is the smallest simplicial subset of $\Delta[k]$ containing $d_i\sigma_k$ for $0\leq i\leq k$. Let $S^k=\Delta[k]/\partial\Delta[k]$. Then the geometric realization $|S^k|$ is homeomorphic to the standard $k$-sphere $S^k$. As a simplicial set, $S^k_n=\{\ast\}$ for $n<k$ and
\begin{equation}\label{equation3.1}
\begin{array}{rcl}
S^k_n&=&\{\ast, (i_0,i_1,\ldots,i_n)\ | \ 0\leq i_0\leq i_1\leq \cdots \leq i_n\leq k \\
&& \quad \quad \quad \textrm{ with } \{0,1,\ldots,k\}=\{i_0,i_1,\ldots,i_n\}\}\\
&=&\{\ast, s_{j_{n-k}}s_{j_{n-k-1}}\cdots s_{j_1}\sigma_k \ | \ 0\leq j_1<j_2<\cdots <j_{n-k}\leq n-1\}\\
\end{array}
\end{equation}
for $n\geq k$. In the first description above, it is required that
each $0\leq j\leq k$ appears at least once in the sequence
$(i_0,\ldots,i_n)$. In this description, we can describe the faces
and degeneracies by removing-doubling coordinates where we
identify the sequence $(i_0,\ldots,i_n)$ to be the base-point of
any one of $0\leq j\leq k$ does not appear in $(i_0,\ldots,i_n)$.
In the second description, we can use the simplicial identities to
describe the faces and degeneracies on $S^k$.

By applying Milnor's construction to $S^k$, we obtain the simplicial group $F[S^k]\simeq \Omega S^{k+1}$ with $F[S^k]_n$ a free group of rank $\binom{n}{k}$. The generators for $F[S^k]_n$ are given in formula~(\ref{equation3.1}) with $\ast=1$.

\subsection{The Simplicial Group $\AP_*$}
There is a canonical simplicial group arising from pure braid
groups systematically investigated in~\cite{BCWW}. We are only
interested in classical Artin pure braids and so we follow the
discussion in~\cite{CW1}. Let $\AP_n=P_{n+1}$ with the face
homomorphism
$$
d_i\colon \AP_n=P_{n+1}\longrightarrow \AP_{n-1}=P_n
$$
given by removing the $(i+1)$st strand of $(n+1)$-strand pure braids and the degeneracy homomorphism
$$
s_i\colon \AP_n=P_{n+1}\longrightarrow \AP_{n+1}=P_{n+2}
$$
given by doubling the $(i+1)$st strand of $(n+1)$-strand pure braids for $0\leq i\leq n$. Then $\AP_*$ forms a simplicial group. Let $A_{i,j}$, $1\leq i<j\leq n+1$, be the standard generators for $\AP_n=P_{n+1}$. Then the face operations in the simplicial group $\mathrm{AP}_*$ are
defined as follows:
%%%%%%%    conventions as in the Skye paper      %%%%%
\begin{equation}\label{equation3.2}
d_t(A_{i,j})  =
\begin{cases}
A_{i-1,j-1}
  & \quad   \textrm{ if } t+1 < i, \\
1
  & \quad   \textrm{ if } t +1 = i, \\
A_{i,j-1}
  & \quad   \textrm{ if } i < t + 1 < j,\\
1
  & \quad   \textrm{ if } t + 1 = j,\\
A_{i,j}
  & \quad   \textrm{ if } t + 1 > j.\\
\end{cases}
\end{equation}
and the degeneracy operations are defined as follows:
\begin{equation}\label{equation3.3}
s_t(A_{i,j}) =
\begin{cases}
A_{i+1,j+1}
  & \quad   \textrm{ if }  t+1 < i,\\
A_{i,j+1}\cdot A_{i+1,j+1}
  & \quad   \textrm{ if } t+1 = i,\\
A_{i,j+1}
  & \quad   \textrm{ if } i < t + 1 < j,\\
A_{i,j}\cdot A_{i,j+1}
  & \quad   \textrm{ if } t + 1 = j, \\
A_{i,j}
  & \quad   \textrm{ if } t + 1 > j. \\
\end{cases}
\end{equation}

Observe that $\AP_1=P_2\cong \Z$ is generated by $A_{1,2}$ with
$d_0A_{1,2}=d_1A_{1,2}=1$. The representing simplicial map
$$
f_{A_{1,2}}\colon S^1\longrightarrow \AP_*
$$
with $f_{\sigma_1}=A_{1,2}$ extends uniquely to a simplicial homomorphism
$$
\Theta\colon F[S^1]\longrightarrow \AP_*.
$$
The following embedding theorem plays an important role for our constructions of simplicial group models for the loop spaces of spheres and Moore spaces.
\begin{thm}~\cite[Theorem 1.2]{CW1}\label{theorem3.1}
The simplicial homomorphism $$\Theta\colon F[S^1]\longrightarrow \AP_*$$ is a simplicial monomorphism.\hfill $\Box$
\end{thm}

\subsection{Simplicial Group Models for $\Omega S^{k}$ with $k\geq 3$}\label{subsection3.3}
Assume that $k\geq 3$. Let $\alpha\in F[S^1]_{k-2}$ such that
\begin{enumerate}
\item[1)] $\alpha\not=1$ and
\item[2)] $d_j\alpha=1$ for all $0\leq j\leq k-2$, that is, $\alpha$ is a Moore cycle.
\end{enumerate}
(\textbf{Note}. We do not assume that $\alpha$ induces a
nontrivial element in $\pi_{k-2}(F[S^1])=\pi_{k-1}(S^2)$. There
are many choices for such an $\alpha$. We will give a particular
choice of $\alpha$ with braided instructions later. For a moment
$\alpha$ is given by any nontrivial Moore cycle.) The representing
simplicial map
$$
f_{\alpha}\colon S^{k-2}\longrightarrow F[S^1]
$$
extends uniquely to a simplicial homomorphism
$$
\tilde f_{\alpha}\colon F[S^{k-2}]\longrightarrow F[S^1]
$$
by the universal property of Milnor's construction.
\begin{lem}\label{lemma3.2}
Let $k\geq 3$ and let $\alpha\not=1\in F[S^1]_{k-2}$ be a Moore cycle. Then the map
$$
\tilde f_{\alpha}\colon F[S^{k-2}]\longrightarrow F[S^1]
$$
is a simplicial monomorphism.
\end{lem}
\begin{proof}
Let $G=\tilde f_{\alpha}(F[S^{k-2}])$ be the image of $\tilde f_{\alpha}$. Then $G$ is a simplicial subgroup of $F[S^1]$. Since $F[S^1]_q$ is a free group, $G_q$ is free group for each $q$. The statement will follow if we can prove that the simplicial epimorphism
$$
\tilde f_{\alpha}\colon F[S^{k-2}]\longrightarrow G
$$
is a simplicial monomorphism. Observe that since each
$F[S^{k-2}]_q$ is a free group which is residually nilpotent, it
suffices to show that the morphism of the associated Lie algebras
induced from the lower central series
$$
L(\tilde f_{\alpha})\colon L(F[S^{k-1}])\longrightarrow L(G)
$$
is a simplicial isomorphism. For each $q$, since both $F[S^{k-2}]_q$ and $G_q$ are free group, their associated Lie algebras are the free Lie algebras generated by their abelianizations. Thus it suffices to show that
$$
\tilde f^{\ab}_{\alpha}\colon F[S^{k-2}]^{\ab}=K(\Z,k-2)\longrightarrow G^{\ab}
$$
is a simplicial isomorphism.

Note that the Moore chain complex of $K(\Z,k-2)$ is given by
$$
N_qK(\Z,k-2)=\left\{
\begin{array}{rcl}
0&\textrm{ if }& q\not=k-2,\\
\Z&\textrm{ if }&q=k-2.\\
\end{array}\right.
$$
Since $\tilde f^{\ab}_{\alpha}\colon K(\Z,k-2)\to G^{\ab}$ is a simplicial epimorphism,
$$
N(\tilde f^{\ab}_{\alpha})\colon N_qK(\Z,k-2)\longrightarrow N_qG^{\ab}
$$
is an epimorphism for any $q$ by Proposition~\ref{proposition2.1}. It follows that $N_qG^{\ab}=0$ for $q\not=k-2$. For $q=k-2$, we have $N_{k-2}G^{\ab}=G_{k-2}=\la \alpha\ra\cong\Z$ with
$$
N(\tilde f^{\ab}_{\alpha})\colon N_{k-2}K(\Z,k-2)\cong\Z \longrightarrow N_{k-2}G^{\ab}\cong\Z
$$
an isomorphism from the definition of $f_{\alpha}$. Thus
$$
N(\tilde f^{\ab}_{\alpha})\colon NF[S^{k-2}]^{\ab}\longrightarrow NG^{\ab}
$$
is an isomorphism. By Proposition~\ref{proposition2.1}, $\tilde
f^{\ab}_{\alpha}\colon F[S^{k-2}]^{\ab}=K(\Z,k-2)\longrightarrow
G^{\ab}$ is a simplicial isomorphism. This finishes the proof.
\end{proof}

Now, by Theorem~\ref{theorem3.1} and Lemma~\ref{lemma3.2}, the composite
$$
\phi_{\alpha}\colon F[S^{k-2}]\rTo^{\tilde f_{\alpha}} F[S^1]\rTo^{\Theta} \AP_*
$$
is a simplicial monomorphism. Define the simplicial group
$\calT(S^k;\alpha)$ to be the free product with amalgamation
defined by the diagram
\begin{diagram}
F[S^{k-2}]&\rTo^{\phi_{\alpha}}& \AP_*\\
\dTo>{\phi_{\alpha}}&&\dTo\\
\AP_*&\rTo&\calT(S^k;\alpha)=\AP_*\ast_{F[S^{k-2}]}\AP_*.\\
\end{diagram}

\begin{thm}\label{theorem3.3}
Let $k\geq 3$ and let $\alpha\not=1\in F[S^1]_{k-2}$ be a Moore cycle. Then the geometric realization of the simplicial group $\calT(S^k;\alpha)$ is homotopy equivalent to $\Omega S^k$.
\end{thm}
\begin{proof}
By Theorems~\ref{theorem2.1}, the classifying space $\bar W\calT(S^k;\alpha)$ is the homotopy push-out of
\begin{diagram}
\bar W F[S^{k-2}]\simeq S^{k-1}&\rTo&\bar W \AP_*\\
\dTo&&\dTo\\
\bar W\AP_*&\rTo& \bar W\calT(S^k;\alpha).\\
\end{diagram}
By~\cite[Theorem 1.1]{CW1}, $\AP_*$ is a contractible simplicial group and so $\bar W\AP_*$ is contractible. It follows that
$$
\bar W\calT(S^k;\alpha)\simeq S^k
$$
and hence the result.
\end{proof}

\subsection{Some Technical Lemmas}\label{subsection3.4}
Recall ~\cite[p. 288-289]{MKS} that \textit{a bracket arrangement of weight $n$} in a group $G$ is a map
$\beta^n\colon G^n\to G$ which is defined inductively as follows:
$$\beta^1=\id_G,\,\, \beta^2(a_1,a_2)=[a_1,a_2]$$
for any $a_1,a_2\in G$, where $[a_1,a_2]=a_1^{-1}a_2^{-1}a_1a_2$. Suppose that the bracket arrangements of weight $k$ are
defined for $1\leq k<n$ with $n\geq 3$. A map $\beta^n\colon G^n\to G$ is called
a bracket arrangement of weight $n$ if $\beta^n$ is the composite
\begin{diagram}
G^n=G^k\times G^{n-k}&\rTo^{\beta^k\times\beta^{n-k}}&G\times G&\rTo^{\beta^2}&G
\end{diagram}
for some bracket arrangements $\beta^k$ and $\beta^{n-k}$ of weight $k$ and
$n-k$, respectively, with $1\leq k<n$. For instance, if $n=3$, there are two bracket arrangements given by
$[[a_1,a_2],a_3]$ and $[a_1,[a_2,a_3]]$.

Let $R_{j}$ be a sequence of subgroups of $G$ for
 $1\leq j\leq n$. The \textit{fat commutator subgroup}  $[[R_{1}, R_2, \dots,R_{n}]]$ is defined to be the subgroup of $G$ generated by all of
the commutators
$$
\beta^t(g_{i_1},\ldots,g_{i_t}),
$$
where
\begin{enumerate}
\item[1)] $1\leq i_s\leq n$;
\item[2)] $\{i_1,\ldots,i_t\}=\{1,\ldots,n\}$, that is each integer in $\{1,2,\cdots,n\}$ appears as at least one of the
integers $i_s$;
\item[3)] $g_j\in R_j$;
\item[4)] $\beta^t$ runs over all of the  bracket arrangements of weight $t$ (with $t\geq n$).
\end{enumerate}
For convenience, let $[[R_1]]=R_1$.

The \textit{symmetric commutator subgroup} $[R_1,
R_2,\ldots,R_n]_S$ defined by
$$
[R_1,R_2,\ldots,R_n]_S=\prod_{\sigma\in
\Sigma_n}[[R_{\sigma(1)},R_{\sigma(2)}],\ldots,R_{\sigma(n)}],
$$
where $[[R_{\sigma(1)},R_{\sigma(2)}],\ldots, R_{\sigma(n)}]$ is the subgroup generated by the left iterated commutators
$$
[[[g_1,g_2],g_3],\ldots,g_n]
$$
with $g_i\in R_{\sigma(i)}$. For convenience, let $[R_1]_S=R_1$. From the definition, the symmetric commutator subgroup is a subgroup of the fat commutator subgroup. In fact they are the same subgroup by the following theorem provided that each $R_j$ is normal.

\begin{lem}~\cite[Theorem 1.1]{LW2}\label{lemma3.4}
Let $R_j$ be any normal subgroup of a group $G$ with $1\leq j\leq n$. Then
$$
[[R_1,R_2,\ldots,R_n]]=[R_1,R_2,\ldots,R_n]_S.
$$\hfill $\Box$
\end{lem}

One can determine the Moore chains and boundaries for the self
free products of $\AP_*$ with a help of the Kurosh theorem on the
structure of subgroups of free products. However, in order to get
this description, we will use another method. We construct a
simplicial free group $\calG$ as follows: For each $n\geq 0$, the
group $\calG_n$ is the free group generated by $ x_{i,j} $ for
$1\leq i<j\leq n+1$. The face and degeneracy operations are given
by formulae~(\ref{equation3.2}) and~(\ref{equation3.3}), where we
replace $A_{i,j}$ by $x_{i,j}$. It is straightforward to check
that the simplicial identities hold. Thus we have a simplicial
group $\calG$.

Now we are going to determine the Moore chains and Moore cycles of the free products of $\calG$. Let $J$ be an index set and let $\calG^{\ast J}=\ast_{\alpha\in J}\calG({\alpha})$, where each $\calG({\alpha})$ is a copy of $\calG$ indexed by an element $\alpha\in J$. For each group $(\calG(\alpha)_n=\calG_n$, let $x_{i,j}(\alpha)$ denote the generator $x_{i,j}$ for $1\leq i<j\leq n+1$. From the definition, $\calG^{\ast J}_n=\ast_{\alpha\in J}\calG(\alpha)_n$ is a free group with a basis given by $\{x_{i,j}(\alpha)\ | \ 1\leq i<j\leq n+1, \ \alpha\in J\}$.

A \textit{basic word} in the group $\calG^{\ast J}_n$ means one of the elements $x_{i,j}(\alpha)^{\pm1}$ for some $\alpha\in J$ and so $1\leq i<j\leq n+1$. Let
$$
w=\beta_t(x_{i_1,j_1}(\alpha_1)^{\pm1}, x_{i_2,j_2}(\alpha_2)^{\pm1},\ldots, x_{i_t,j_t}(\alpha_t)^{\pm1})
$$
be a $t$-fold iterated commutator on basic words, where the bracket $\beta_t(\ \cdots\ )$ is any bracket arrangement. Define
$$
\Index(w)=\{i_1,j_1,i_2,j_2,\ldots,i_t,j_t\}\subseteq \{1,2,\ldots,n+1\}.
$$
(\textbf{Note.} In our definition, $\Index(w)$ is only well-defined for commutators with entries from basic words.)

For each pair $1\leq i<j\leq n+1$, let
$$
R_{i,j}^J=\la x_{i,j}(\alpha)\ | \ \alpha\in J\ra^{\calG^{\ast J}}
$$
be the normal closure of the elements $x_{i,j}(\alpha)$, $\alpha\in J$, in the group $\calG^{\ast J}$. For a subset $T\subseteq \{1,2,\ldots,n+1\}$, define
$$
R[T]=\prod_{T\subseteq\{i_1,j_1,i_2,j_2,\ldots,i_t,j_t\}}[[R_{i_1,j_1},R_{i_2,j_2}],\ldots, R_{i_t,j_t}]
$$
be the product of the iterated commutator subgroup of $R_{i,j}$'s
such that each number in $T$ occurs at least once in the indices
of $R_{i,j}$'s. (Here if $t=1$, then we let commutator subgroup
$[R_{i_1,j_1}]=R_{i_1,j_1}$ by convention.) In the case that
$T=\{1,2,\ldots,n+1\}$, we denote
$$
[R_{i,j}\ | \ 1\leq i<j\leq n+1]_S
$$
by $R[1,2,\ldots,n+1]$.

\begin{lem}\label{lemma3.5}
Let $\calG^{\ast J}$ be the self free product of $\calG$ over a set $J$. Then
\begin{enumerate}
\item[1)] The Moore chains $N_n\calG^{\ast J}=R[2,3,\ldots,n+1]$.
\item[2)] The Moore cycles $\calZ_n\calG^{\ast J}=R[1,2,3,\ldots,n+1]=[R_{i,j}\ | \ 1\leq i<j\leq n+1]_S$.
\item[3)] The Moore boundaries $\calB_n\calG^{\ast J}=\calZ_n\calG^{\ast J}=R[1,2,3,\ldots,n+1]=[R_{i,j}\ | \ 1\leq i<j\leq n+1]_S$.
\end{enumerate}
\end{lem}
\begin{proof}
For assertions (1) and (2), the direction
$$
R[2,3,\ldots,n+1]\leq N_n\calG^{\ast J}\textrm{ and } R[1,2,3,\ldots,n+1]\leq \calZ_n\calG^{\ast J}
$$
can be easily checked as follows. From equation~(\ref{equation3.2}), we have $d_kx_{i,j}(\alpha)=1$ for $\alpha\in J$ if $k+1=i$ or $j$. Thus
$$
R_{i,j}\leq \Ker(d_k\colon \calG^{\ast J}_n\to \calG^{\ast J}_{n-1})
$$
if $k+1=i$ or $j$. Thus
$$
[[R_{i_1,j_1},R_{i_2,j_2}],\ldots, R_{i_t,j_t}]\leq N_n\calG^{\ast J}=\bigcap_{k=1}^n \Ker(d_k\colon \calG^{\ast J}_n\to \calG^{\ast J}_{n-1})
$$
if $\{2,3,\ldots,n+1\}\subseteq \{i_1,j_1,i_2,j_2,\ldots,
i_t,j_t\}$ since each $d_k$, $1\leq k\leq n$, sends one of entries
$R_{i_s,j_s}$ in this (iterated) commutator subgroup to the
trivial group. It follows that $R[2,3,\ldots,n+1]\leq
N_n\calG^{\ast J}$. Similarly
$R[1,2,\ldots,n+1]\leq\calZ_{n+1}\calG^{\ast J}$. Thus the main
point is to prove that
\begin{equation}\label{equation3.4}
N_n\calG^{\ast J}\leq R[2,3,\ldots,n+1] \textrm{ and } \calZ_n\calG^{\ast J}\leq R[1,2,\ldots,n+1].
\end{equation}
If $n=1$, then $R[2]=R[1,2]=\calG^{\ast J}_1$ because $\calG^{\ast J}_1$ is generated by $x_{1,2}(\alpha)$ for $\alpha\in J$. In this case, the identity that $N_1\calG^{\ast J}=\calZ_1\calG^{\ast J}=R[2]=R[1,2]=\calG^{\ast J}_1$. Thus we may assume that $n\geq 2$.

We first consider the last face operation
$$
d_n\colon \calG^{\ast J}_n\longrightarrow \calG^{\ast J}_{n-1}.
$$
Let $K_n=\Ker(d_n\colon \calG^{\ast J}_n\to \calG^{\ast J}_{n-1})$. From equation~(\ref{equation3.2}),
$$
d_n(x_{i,j}(\alpha))=\left\{
\begin{array}{lcl}
1&\textrm{ if } &1\leq i<j=n+1,\\
x_{i,j}(\alpha)&\textrm{ if }& 1\leq i<j\leq n.\\
\end{array}\right.
$$
Observe that the basis of $\calG^{\ast J}_n$ is given by the
disjoint union of the basis of $\calG^{\ast J}_{n-1}$ with the set
$\{x_{i,n+1}(\alpha) \ | \ 1\leq i<n+1,\ \alpha\in J\}$.
By~\cite[Proposition 3.3]{Wu2}, a basis for the free group $K_n$
is given by the subset $X_n$ of $\calG^{\ast J}_n$ consisting of
all of the following iterated commutators on basic words
\begin{equation}\label{equation3.5}
 w=[[[x_{i,n}(\alpha_0), x_{i_1,j_1}^{\epsilon_1}(\alpha_1)], x_{i_2,j_2}^{\epsilon_2}(\alpha_2)],\ldots, x_{i_t,j_t}^{\epsilon_t}(\alpha_t)],
\end{equation}
where
\begin{enumerate}
\item[1)] $t\geq0$ (Here if $t=0$, then $w=[x_{i,n}(\alpha_0)]=x_{i,n}(\alpha_0)$.)
\item[2)] $\epsilon_s=\pm1$ for $1\leq s\leq t$,
\item[3)] $1\leq i_s<j_s\leq n$ for $1\leq s\leq t$,
\item[4)] $\alpha_s\in J$ for $0\leq s\leq t$ and
\item[5)] the word $x_{i_1,j_1}^{\epsilon_1}(\alpha_1)x_{i_2,j_2}^{\epsilon_2}(\alpha_2) \cdots x_{i_t,j_t}^{\epsilon_t}(\alpha_t)$ is an irreducible word in the group $\calG^{\ast J}_{n-1}\leq \calG^{\ast J}_n$.
\end{enumerate}

Next we consider the face operation $d_k$ restricted to $K_n$ for
$0\leq k<n$. From the $\Delta$-identity $d_kd_n=d_{n-1}d_k$ for
$1\leq k\leq n-1$, we have the commutative diagram of short exact
sequence of groups
\begin{equation}\label{equation3.6}
\begin{diagram}
K_n=F(X_n)&\rInto& \calG^{\ast J}_n&\rOnto^{d_n}&\calG^{\ast J}_{n-1}\\
\dTo>{d_k|_{K_n}}&&\dTo>{d_{k}}&&\dTo>{d_k}\\
K_{n-1}=F(X_{n-1})&\rInto&G^{\ast J}_{n-1}&\rOnto^{d_{n-1}}&\calG^{\ast J}_{n-2}\\
\end{diagram}
\end{equation}
for $1\leq k\leq n-1$. Consider $d_kw$ for $w\in X_n$. From equation~(\ref{equation3.2}), $d_kx_{i,n}(\alpha)$ is given by the following table
$$
\begin{array}{lcl}
d_k&&\left(
\begin{array}{ccccccc}
x_{1,n}(\alpha)&\cdots& x_{k-1,n}(\alpha) & x_{k,n}(\alpha)& x_{k+1,n}(\alpha)&\cdots& x_{n-1,n}(\alpha)\\
\downarrow&  &\downarrow&\downarrow&\downarrow& &\downarrow\\
x_{1,n-1}(\alpha)&\cdots&x_{k-1,n}(\alpha)&1&x_{k,n-1}(\alpha)&\cdots&x_{n-2,n-1}(\alpha)\\
\end{array}\right).
\end{array}
$$

We now start to prove statement~(\ref{equation3.4}). Let
$$
X_n(k)=\{w\in X_n\ | \ k+1\in\Index(w)\}
$$
for $0\leq k\leq n-1$. If $w\in X_n(k)$, then $d_kw=1$ as $d_k$
sends one of the entries in the commutator $w$ to $1$. Let $w\in
X_n\smallsetminus X_n(k)$ be written as in~(\ref{equation3.5}).
Then $i\not=k+1$ and $k+1\not\in\{i_1,j_1,\ldots,i_t,j_t\}$. From
the above table, $d_kx_{i,n}(\alpha_0)=x_{i,n-1}(\alpha_0)$ for
$i<k+1$ and $x_{i-1,n-1}(\alpha_0)$ for $i>k+1$. For other entries
$x^{\epsilon_s}_{i_s,j_s}(\alpha_s)$, we have
$$
d_k(x^{\epsilon_s}_{i_s,j_s}(\alpha_s))=\left\{
\begin{array}{lcl}
x^{\epsilon_s}_{i_s-1,j_s-1}(\alpha_s)&\textrm{ if }& k+1<i_s,\\
x^{\epsilon_s}_{i_s,j_s-1}(\alpha_s)&\textrm{ if }& i_s<k+1<j_s,\\
x^{\epsilon_s}_{i_s,j_s}(\alpha_s)&\textrm{ if }& k+1>j_s.\\
\end{array}\right.
$$
Observe that
$$
d_k\colon \{x_{i,j}(\alpha) \ | \ \alpha\in J, \ 1\leq i<j\leq n \textrm{ and } k+1\not=i,j\}\longrightarrow \{x_{i,j}(\alpha) \ |\alpha\in J, \ 1\leq i<j\leq n-1\}
$$
is a bijection. The restriction of $d_k$ in the subgroup
$$
d_k|\colon F(x_{i,j}(\alpha) \ | \ \alpha\in J, \ 1\leq i<j\leq n \textrm{ and } k+1\not=i,j)\longrightarrow \calG^{\ast J}_{n-2}
$$
is an isomorphism. Since the word
$$
x_{i_1,j_1}^{\epsilon_1}(\alpha_1)x_{i_2,j_2}^{\epsilon_2}(\alpha_2) \cdots x_{i_t,j_t}^{\epsilon_t}(\alpha_t)\in F(x_{i,j}(\alpha) \ | \ \alpha\in J, \ 1\leq i<j\leq n \textrm{ and } k+1\not=i,j)
$$
is irreducible,
the word
$$
d_k(x_{i_1,j_1}^{\epsilon_1}(\alpha_1)x_{i_2,j_2}^{\epsilon_2}(\alpha_2) \cdots x_{i_t,j_t}^{\epsilon_t}(\alpha_t))=(d_kx_{i_1,j_1})^{\epsilon_1}(\alpha_1)(d_kx_{i_2,j_2})^{\epsilon_2}(\alpha_2) \cdots (d_kx_{i_t,j_t})^{\epsilon_t}(\alpha_t)
$$
is irreducible in $\calG^{\ast J}_{n-2}\leq \calG^{\ast J}_{n-1}$. It follows that $d_kw\in X_{n-1}$ for each $w\in X_n\smallsetminus X_n(k)$ and the function
$$
d_k\colon X_n\smallsetminus X_n(k)\longrightarrow X_{n-1}
$$
is a bijection. This allows us to apply the algorithm in~\cite[Section 3]{Wu2} to
$$
d_k|\colon K_n=F(X_n)\longrightarrow K_{n-1}=F(X_{n-1})
$$
for $0\leq k\leq n-1$ in diagram~(\ref{equation3.6}) and so, by~\cite[Theorem 3.4]{Wu2}, the Moore chains
$$
N_n\calG^{\ast J}=\bigcap_{k=1}^{n-1}\Ker(d_k|\colon K_n\to K_{n-1})
$$
are generated by certain iterated commutators
\begin{equation}\label{equation3.7}
w=\beta_t(x_{i_1,j_1}(\alpha_1)^{\pm1}, x_{i_2,j_2}(\alpha_2)^{\pm1},\ldots, x_{i_t,j_t}(\alpha_t)^{\pm1})
\end{equation}
with $\{2,3,\ldots,n+1\}\in \Index(w)$ and the Moore cycles
$$
\calZ_n\calG^{\ast J}=\bigcap_{k=0}^{n-1}\Ker(d_k|\colon K_n\to K_{n-1})
$$
is generated by certain iterated commutators
\begin{equation}\label{equation3.8}
w=\beta_t(x_{i_1,j_1}(\alpha_1)^{\pm1}, x_{i_2,j_2}(\alpha_2)^{\pm1},\ldots, x_{i_t,j_t}(\alpha_t)^{\pm1})
\end{equation}
with $\{1,2,\ldots,n+1\}\in \Index(w)$. (\textbf{Note.} The
commutator $w$ in ~(\ref{equation3.7}) or ~(\ref{equation3.8}) may
not be in the standard form from left to right.) Since each entry
$x_{i_s,j_s}(\alpha_s)^{\pm1}$ belongs to $R_{i_s,j_s}$, the
commutator $w$ in ~(\ref{equation3.7}) or ~(\ref{equation3.8})
lies in the fat commutator subgroup
$[[R_{i_1,j_1},R_{i_2,j_2},\ldots,R_{i_t,j_t}]]$ and so, by
Lemma~\ref{lemma3.4},
$$
w\in \prod_{\sigma\in \Sigma_t} [[R_{i_{\sigma(1)},j_{\sigma(1)}},R_{i_{\sigma(2)},j_{\sigma(2)}}],\ldots,R_{i_{\sigma(t)},j_{\sigma(t)}}]\leq R[s,s+1,\ldots, n+1],
$$
where $s=2$ in the case of ~(\ref{equation3.7}) and $s=1$ in the
case of ~(\ref{equation3.8}). This finishes the proof of
statement~(\ref{equation3.4}) and hence assertion (1) and (2).

\noindent (3). By assertion (2),
$$
\calZ_n\calG^{\ast J}=\prod_{\{1,2,\ldots,n+1\}\subseteq\{i_1,j_1,i_2,j_2,\ldots,i_t,j_t\}}[[R_{i_1,j_1},R_{i_2,j_2}],\ldots, R_{i_t,j_t}].
$$
From equation~\ref{equation3.2}, we have $d_0x_{i+1,j+1}(\alpha)=x_{i,j}(\alpha)$ for $1\leq i<j\leq n+1$ and $\alpha\in J$. Thus
$$
d_0(R_{i+1,j+1})=R_{i,j}
$$
for $1\leq i<j\leq n+1$. Given a factor $[[R_{i_1,j_1},R_{i_2,j_2}],\ldots, R_{i_t,j_t}]$ in $\calZ_n\calG^{\ast J}$ with $\{1,2,\ldots,n+1\}\subseteq\{i_1,j_1,i_2,j_2,\ldots,i_t,j_t\}$, we have
$$
d_0([[R_{i_1+1,j_1+1},R_{i_2+1,j_2+1}],\ldots, R_{i_t+1,j_t+1}])=[[R_{i_1,j_1},R_{i_2,j_2}],\ldots, R_{i_t,j_t}].
$$
Since $\{2,3,\ldots, n+2\}\subseteq \{i_1+1,j_1+1,i_2+1,j_2+1,\ldots,i_t+1,j_t+1\}$, the subgroup
$$
\prod_{\{1,2,\ldots,n+1\}\subseteq\{i_1,j_1,i_2,j_2,\ldots,i_t,j_t\}}[[R_{i_1+1,j_1+1},R_{i_2+1,j_2+1}],\ldots, R_{i_t+1,j_t+1}]\leq N_{n+1}\calG^{\ast J}
$$
with
$$
d_0\left(\prod_{\{1,2,\ldots,n+1\}\subseteq\{i_1,j_1,i_2,j_2,\ldots,i_t,j_t\}}[[R_{i_1+1,j_1+1},R_{i_2+1,j_2+1}],\ldots, R_{i_t+1,j_t+1}]\right)=\calZ_n\calG^{\ast J}.
$$
It follows that $\calZ_n\calG^{\ast J}\leq \calB_n\calG^{\ast J}$.
Assertion (3) follows and this finishes the proof.
\end{proof}

The following lemma states that $R[1,2,\ldots,n+1]$ can be given by the product of a finite collection of commutator subgroups.
\begin{lem}\label{lemma3.6}
The subgroup $R[1,2,\ldots,n+1]$ of $\calG^{\ast J}_n$ is the product of the following commutator subgroups
$$
[[R_{i_1,j_1},R_{i_{2},j_{2}}],\ldots, R_{i_{t},j_{t}}],
$$
where
\begin{enumerate}
\item[1)] $\{i_1,j_1,i_2,j_2,\ldots,i_t,j_t\}=\{1,2,\ldots,n+1\}$
and \item[2)] $\{i_1,j_1,i_2,j_2,\ldots,i_t,j_t\}\smallsetminus
\{i_p,j_p\}\not=\{1,2,\ldots,n+1\}$ for any $1\leq p\leq t$.
\end{enumerate}
\end{lem}
\begin{proof}
Let $H$ be the product of the commutator subgroups given in the statement. Clearly $H\leq R[1,2,\ldots,n+1]$. Now consider the factor
$$
[[R_{i_1,j_1},R_{i_2,j_2}],\ldots,R_{i_t,j_t}]
$$
with $\{1,2,\ldots,n+1\}=\{i_1,j_1,i_2,j_2,\ldots,i_t,j_t\}$ in $R[1,2,\ldots,n+1]$. If there exists $1\leq p\leq t$ such that
$$
\{i_1,j_1,i_2,j_2,\ldots,i_t,j_t\}\smallsetminus
\{i_p,j_p\}=\{1,2,\ldots,n+1\},
$$
since $[[R_{i_1,j_1}, R_{i_2,j_2}],\ldots,R_{i_{p-1},j_{p-1}}]$ is
normal, we have
$$
[[[R_{i_1,j_1}, R_{i_2,j_2}],\ldots,R_{i_{p-1},j_{p-1}}], R_{i_p,j_p}]\leq [[R_{i_1,j_1}, R_{i_2,j_2}],\ldots,R_{i_{p-1},j_{p-1}}].
$$
(If $p=1$, then we use $[R_{i_1,j_1}, R_{i_2,j_2}]\leq R_{i_2,j_2}.$) It follows that
$$
[[R_{i_1,j_1},R_{i_2,j_2}],\ldots,R_{i_t,j_t}]\leq [[R_{i_1,j_1},R_{i_2,j_2}],\ldots,\hat R_{i_p,j_p},\ldots, R_{i_t,j_t}]
$$
with $\{i_1,j_1,i_2,j_2,\ldots,i_t,j_t\}\smallsetminus \{i_p,j_p\}=\{1,2,\ldots,n+1\}$. By repeating this process by removing surplus entries, we have
$$
[[R_{i_1,j_1},R_{i_2,j_2}],\ldots,R_{i_t,j_t}]\leq H
$$
and hence the result.
\end{proof}

The following simple result is well-known and follows from the
structure of normal forms of free products with amalgamation (for
the proof see, for example, \cite{Emb}):
\begin{lem}\label{cent}
Let $G=G_1*_AG_2$ be a free product with amalgamation such that
$G_1\neq A$ and $G_2\neq A$. Then $Z(G)\leq Z(A)$.
\end{lem}

\subsection{Proof of Theorem~\ref{theorem1.2}}
We use our simplicial group model $\calT(S^k,\alpha)$ for $\Omega S^k$.
Consider the construction of the subgroup $Q_{n,k}$ of $P_n$. By the definition of the simplicial group $\AP_*$, the iterated degeneracy operations on $\AP_1=P_2$ are given by the cabling and so the elements $x_1,\ldots,x_{k-2}$ in Step 1 are the canonical basis for the subgroup
$$
\Theta(F[S^1]_{k-2})\leq \AP_{k-2}=P_{k-1}.
$$
Since $d_ix_i=d_ix_{i+1}$ for $1\leq i\leq k-3$ and $d_0x_1=d_{k-2}x_{k-2}=1$, we have $d_i\alpha_k=1$ for $0\leq i\leq k-2$. It follows that $\alpha_k$ is a Moore cycle in $F[S^1]_{k-2}$ with $\alpha_k\not=1$. The elements $y_j$, $1\leq j\leq \binom{n-1}{k-2}$, are standard basis for the subgroup
$$
\phi_{\alpha_k}(F[S^{k-2}]_{n-1})\leq \AP_{n-1}=P_n
$$
since they are obtained by cabling on $\alpha_k$. It follows that
$$
P_n\ast_{Q_{n,k}}P_n=\left(\AP_*\ast_{F[S^{k-2}]}\AP_*\right)_{n-1}=\calT(S^k;\alpha_k)_{n-1}.
$$
Theorem~\ref{theorem1.2} is a special case of the following slightly more general statement.

\begin{thm}\label{theorem3.7}
Let $k\geq 3$ and let $\alpha\not=1\in F[S^1]_{k-2}$ be a Moore cycle. Then the simplicial group $\calT(S^k;\alpha)\simeq \Omega S^k$ has the following properties:
\begin{enumerate}
\item[1)] In the group $\calT(S^k;\alpha)_{n-1}=P_n\ast_{F[S^{k-2}]_{n-1}}P_n$, the Moore boundaries
$$\calB_{n-1}\calT(S^k;\alpha)=[R_{i,j}\ | \ 1\leq i<j\leq n]_S.$$
\item[2)] The homotopy group $\pi_n(S^k)\cong \pi_{n-1}(\Omega S^k)\cong \pi_{n-1}(\calT(S^k;\alpha))$ is isomorphic to the center of the group
$$
\calT(S^k;\alpha)_{n-1}/\calB_{n-1}\calT(S^k;\alpha)=(P_n\ast_{F[S^{k-2}]_{n-1}}P_n)/[R_{i,j}\ | \ 1\leq i<j\leq n]_S.
$$
for any $n$ if $k>3$ and any $n\not=3$ if $k=3$.
\end{enumerate}
\end{thm}
\begin{proof}
(1). By definition, the simplicial group $\calT(S^k;\alpha)$ is
given by the free product with amalgamation
$\AP_*\ast_{F[S^{k-2}]}\AP_*$. Thus $\calT(S^k;\alpha)$ is a
simplicial quotient group of the free product $\AP_*\ast \AP_*$.
Let $\calG$ be the simplicial group given in
Subsection~\ref{subsection3.4}. Then $\AP_*$ is a simplicial
quotient group of $\calG$. It follows that there is a simplicial
epimorphism
$$
g\colon \calG\ast\calG\rOnto \calT(S^k;\alpha).
$$
By Proposition~\ref{proposition2.1},
$$
N(g)=g|\colon N_n(\calG\ast\calG)\longrightarrow N_n(\calT(S^k;\alpha)
$$
is an epimorphism and so
$$
\begin{array}{rcl}
\calB_{n-1}(\calT(S^k;\alpha))&=&d_0(N_n(\calT(S^k;\alpha)))\\
&=&d_0(g(N_n(\calG\ast\calG)))\\
&=&g(d_0(N_n(\calG\ast\calG)))\\
&=&g(\calB_{n-1}(\calG\ast\calG)).\\
\end{array}
$$
Assertion (1) follows from Lemma~\ref{lemma3.5}.

(2). \textbf{Case I.} $k>3$. Since
$\calT(S^k;\alpha)_q=P_{q+1}\ast_{F[S^{k-2}]_q}P_{q+1}$ is a free
product with amalgamation, the center $Z(\calT(S^k;\alpha)_q)\leq
Z(F[S^{k-2}]_q)=\{1\}$ for $q\geq k-1$ by Lemma \ref{cent}. For
$q=k-2$, then $Z(\calT(S^k;\alpha)_{k-2})\leq
F[S^{k-2}]_{k-2}=\la\alpha\ra=\Z$ by Lemma \ref{cent}. Since
$\alpha$ is Moore cycle, $\alpha$ is a Brunnian braid in
$P_{k-1}$. Recall that the center of $P_{k-1}$ is given by the
full-twist braid $\Delta^2$~\cite{Chow} with the property that, by
removing any one of the strands of $\Delta^2$, it becomes a
generator for the center of $P_{k-1}$ and $d_i\Delta^2\not=1$ for
$k>3$. Since $\alpha$ is Brunnian braid, any power
$\alpha^m\not\in Z(P_{k-1})$ for $m\not=0$. It follows that
$\alpha^m\not\in Z(P_{k-1}\ast_{F[S^{k-2}]_{k-2}}P_{k-1})$ for
$m\not=0$. Thus $Z(\calT(S^k;\alpha)_{k-2})=\{1\}$. For $q< k-2$,
$\calT(S^k,\alpha)_q$ is a free product and so
$Z(\calT(S^k,\alpha)_q)=\{1\}$, where for the low cases,
$\calT(S^k;\alpha)_1=P_2\ast P_2$ is a free group of rank $2$ and
$\calT(S^k;\alpha)_0=\{1\}$. Thus the center
$Z(\calT(S^k;\alpha)_q=\{1\}$ for all $q\geq0$. It follows
from~\cite[Proposition 2.14]{Wu2} that
$$
\pi_q(\calT(S^k;\alpha))\cong Z(\calT(S^k;\alpha)_q/\calB_q(\calT(S^k;\alpha))
$$
for $q\geq1$. This isomorphism also holds for $q=0$ because $\calT(S^k;\alpha)_0=\{1\}$.

\textbf{Case II.} $k=3$. By the same arguments as above, we have $Z(\calT(S^3;\alpha)_q)=\{1\}$ for $q\geq 2$. By~\cite[Proposition 2.14]{Wu2}, we have
\begin{equation}\label{equation3.9}
\pi_q(\calT(S^3;\alpha))\cong Z(\calT(S^3;\alpha)_q/\calB_q(\calT(S^3;\alpha))
\end{equation}
for $q\geq3$. We only need to check that this isomorphism also
holds for the cases $q=0,1$. (The case that $q=2$ is the
exceptional case, which is excluded in the statement.) When $q=0$,
both sides are trivial groups. Consider the case $q=1$. Note that
$\AP_1=P_2\cong\Z$ generated by $A_{12}$. Since $\alpha$ is not
trivial, it is given by a nontrivial power of $A_{12}$. Let
$\alpha=A_{12}^m$ for some $m\not=0$. Then $\calT(S^3;\alpha)_1$
is given by the pushout diagram
\begin{diagram}
P_2=\Z&\rTo^{m}&P_2=\Z\\
\dTo>{m}&&\dTo\\
P_2=\Z&\rTo&\calT(S^3;\alpha)_1.\\
\end{diagram}
Since $R_{1,2}=\la A_{1,2},A'_{1,2}\ra^{\calT(S^3;\alpha)_1}=\calT(S^3,\alpha)_1$ because $\calT(S^3;\alpha)_1$ is generated by $A_{1,2}$ and $A'_{1,2}$, we have
$$
\calB_1(\calT(S^3;\alpha))=\calT(S^3;\alpha)_1
$$
and so
$$
Z(\calT(S^3;\alpha)_1/\calB_1(\calT(S^3;\alpha))=\calT(S^3;\alpha)_1/\calB_1(\calT(S^3;\alpha)=\{1\}.
$$
On the other hand,
$$
\pi_1(\calT(S^3;\alpha))=\pi_1(\Omega S^3)=\pi_2(S^3)=\{1\}.
$$
Thus isomorphism~(\ref{equation3.9}) holds for $q=1$. This finishes the proof.
\end{proof}

\begin{example}\label{example3.8}
{\rm In this example, we discuss the exceptional case by
determining the center of the group:
$$
G=(P_3\ast_{F[S^{1}]_{2}}P_3)/[R_{i,j}\ | \ 1\leq i<j\leq 3]_S,
$$
where $\alpha=A_{1,2}^m$ with some $m\not=0$. By definition, the subgroup $F[S^1]_2\leq P_3$ is generated by $x_1=s_1\alpha_3=(A_{1,3}A_{2,3})^m$ and $x_2=s_0\alpha_3=(A_{1,2}A_{1,3})^m.$ Thus the free product with amalgamation $P_3\ast_{F[S^{1}]_{2}}P_3$ is given as the quotient group of $P_3\ast P_3$ by the new relations:
\begin{equation}\label{equation3.10}
(A_{1,3}A_{2,3})^m=(A'_{1,3}A'_{2,3})^m\textrm{ and } (A_{1,2}A_{1,3})^m=(A'_{1,2}A'_{1,3})^m.
\end{equation}
Consider the subgroup $[R_{i,j}\ | \ 1\leq i<j\leq 3]_S$ of $P_3\ast_{F[S^{1}]_{2}}P_3$. Observe that
$$
[A_{1,2},A_{1,3}],[A_{1,2},A_{2,3}],[A_{1,3},A_{2,3}] \in [R_{i,j}\ | \ 1\leq i<j\leq 3]_S,
$$
the subgroup $\la A_{1,2},A_{1,3},A_{2,3}\ra$ is abelian in $G$. Similarly the subgroup $\la A'_{1,2},A'_{1,3}, A'_{2,3}\ra$ is abelian in $G$. Thus $(A'_{1,2}A'_{1,3})^m=(A'_{1,2})^m(A'_{1,3})^m$ in $G$ and from equation~(\ref{equation3.10})
$$
(A'_{1,2})^m=(A_{1,2})^m(A_{1,3})^m(A'_{1,3})^{-m}.
$$
It follows that $A'_{1,2}$ commutes with $A_{1,2}$ since $A_{1,2}$
commutes with $(A_{1,2})^m, (A_{1,3})^m$ and $(A'_{1,3})^m$. From
this, we conclude that $(A'_{1,2})^m\in Z(G)$ because $A'_{1,2}$
commutes with all of the generators for $G$. Similarly
$(A_{1,2})^m, (A_{1,3})^m, (A_{2,3})^m, (A'_{1,3})^m,
(A'_{2,3})^m\in Z(G)$. Thus the subgroup
\begin{equation}\label{equation3.11}
H=\la ((A_{1,2})^m, (A_{1,3})^m, (A_{2,3})^m, (A'_{1,2})^m,  (A'_{1,3})^m, (A'_{2,3})^m\ra\leq Z(G).
\end{equation}
Let
$$
G'=(\Z(A_{1,2})/m\ast \Z(A'_{1,2})/m)\times (\Z(A_{1,3})/m\ast
\Z(A'_{1,3})/m)\times (\Z(A_{2,3})/m\ast \Z(A'_{2,3})/m)
$$
and let $\phi\colon P_3\ast P_3\to G'$ be the canonical quotient
homomorphism defined by sending generators to generators. Then
$$
\phi(x_1)=\phi(x_2)=1.
$$
Moreover $\phi([R_{i,j}\ | \ 1\leq i<j\leq 3]_S)=1$ with $\phi(H)=1$ and so $\phi$ induces an epimorphism $\bar\phi$ in the following diagram:
\begin{diagram}
P_3\ast P_3&\rOnto^{\phi}& G'=(\Z/m\ast\Z/m)\times (\Z/m\ast\Z/m)\times (\Z/m\ast\Z/m)\\
\dOnto>{q}&\ruOnto^{\bar\phi}&\\
G/H.&&\\
\end{diagram}
On the other hand, the group homomorphism
$$
\Z(A_{1,2})\ast \Z(A'_{1,2})\longrightarrow G/H
$$
factors through the quotient $\Z(A_{1,2})/m\ast \Z(A'_{1,2})/m$.
Similarly there are canonical group homomorphisms from
$\Z(A_{1,3})/m\ast \Z(A'_{1,3})/m$ and $\Z(A_{1,2})/m\ast
\Z(A'_{1,2})/m$ to $G/H$. Since the subgroup $\la
A_{1,2},A'_{1,2}\ra$, $\la A_{1,3},A'_{1,3}\ra$ and $\la
A_{2,3},A'_{2,3}\ra$ commute with each other in the group $G$,
there is a group epimorphism
$$
\psi\colon G'\to G/H
$$
such that $\bar\phi\circ \psi=\id_{G'}$ since all of generators of
$G/H$ lie in the image of $\psi$. It follows that
$$
G/H\cong G'=(\Z/m\ast\Z/m)\times (\Z/m\ast\Z/m)\times (\Z/m\ast\Z/m).
$$
Since $Z(G')=\{1\}$, $Z(G/H)=\{1\}$ and so
$$
Z(G)\leq H.
$$
Together with equation~(\ref{equation3.11}), we have $Z(G)=H\cong \Z^{\oplus 4}$. \hfill $\Box$
}\end{example}

\vspace{.5cm}
\section{Description of Homotopy Groups of the Moore Spaces $M(\Z/q,k)$ with $k\geq3$}\label{section4}
\vspace{.5cm} In this section, we give an explicit combinatorial
description of the homotopy groups of the Moore spaces $M(\Z/q,k)$
with $k\geq 3$. This description highlights our methodology for
giving combinatorial descriptions of homotopy groups using free
products of braid groups.

\subsection{An Embedding of $F[S^{k-1}]$ into $\calT(S^k;\alpha)$ for Moore Boundaries $\alpha$}
Let $\tilde \alpha\in N_{k-1}F[S^1]$ with $d_0\tilde\alpha\not=1$. We are going to construct a simplicial monomorphism $F[S^{k-1}]\to \calT(S^k;d_0\tilde\alpha)$, which is also a homotopy equivalence.

Let
$$
f_{\tilde\alpha}\colon \Delta[k-1]\longrightarrow F[S^1]
$$
be the representing map of the element $\tilde\alpha$ with $f_{\tilde \alpha}(\sigma_k)=\tilde\alpha$, where $\sigma_k=(0,1,\ldots,k-1)\in \Delta[k-1]$. Let $\Lambda^0[k-1]$ be the simplicial subset of $\Delta[k-1]$ generated by $d_j\sigma_{k-1}$ for $j>0$ and let
$$
\bar\Delta[k-1]=\Delta[k-1]/\Lambda^0[k-1].
$$
Since $d_j\tilde\alpha=1$ for $j>0$, the simplicial map $f_{\tilde\alpha}$ factors through the simplicial quotient $\bar\Delta[k-1]$. Let
\begin{equation}\label{equation4.1}
\bar f_{\tilde\alpha}\colon \bar\Delta[k-1]\longrightarrow F[S^1]
\end{equation}
be the resulting simplicial map with $\bar f_{\tilde\alpha}(\sigma_{k-1})=\tilde\alpha$. By the universal property of Milnor's construction, there exists a unique simplicial homomorphism
\begin{equation}\label{equation4.2}
\theta_{\tilde\alpha}\colon F[\bar\Delta[k-1]]\longrightarrow F[S^1]
\end{equation}
such that
$\theta_{\tilde\alpha}|_{\bar\Delta[k-1]}=\bar f_{\tilde\alpha}$.

\begin{lem}\label{lemma4.1}
The simplicial group $F[\bar\Delta[k-1]]$ is contractible and the map
$$
\theta_{\tilde\alpha}\colon F[\bar\Delta[k-1]]\longrightarrow F[S^1]
$$
is a simplicial monomorphism.
\end{lem}
\begin{proof}
Recall~\cite{Curtis} that the geometric realization $|\Delta[k-1]|$ is the standard $(k-1)$-simplex $\Delta^{k-1}$ and $|\Lambda^0[k]|$ is the union of all faces of $\Delta^{k-1}$ except the first face. Thus both $|\Delta[k-1]|$ and $|\Lambda^0[k]|$ are contractible and so is $|\bar\Delta[k-1]|=|\Delta[k-1]/\Lambda^0[k-1]|$. It follows that
$$
|F[\bar \Delta[k-1]]|\simeq\Omega\Sigma|\bar\Delta[k-1]|
$$
is contractible.

The proof of the statement regarding $\theta_{\tilde\alpha}$ is
similar to that of Lemma~\ref{lemma3.2}. The image
$\theta_{\tilde\alpha}(F[\bar\Delta[k-1]])$ is a simplicial free
group because it is a simplicial subgroup of the simplicial free
group $F[S^1]$. Following the lines in the proof of
Lemma~\ref{lemma3.2},  for checking that
$\theta_{\tilde\alpha}\colon F[\bar\Delta[k-1]]\to
\theta_{\tilde\alpha}(F[\bar\Delta[k-1]])$ is a simplicial
monomorphism, it suffices to show that
$$
N\theta_{\tilde\alpha}^{\ab}\colon N_*F[\bar\Delta[k-1]]^{\ab}\longrightarrow N_*\theta_{\tilde\alpha}(F[\bar\Delta[k-1]])
$$
is an isomorphism. This follows directly from the computations that
$$
N_qF[\bar\Delta[k-1]]^{\ab}=\left\{
\begin{array}{lcl}
\Z(\sigma_{k-1})&\textrm{ for }& q=k-1,\\
\Z(d_0\sigma_{k-1})&\textrm{ for }& q=k-2,\\
0&&\textrm{otherwise},\\
\end{array}\right.
$$
$$
N_q\theta_{\tilde\alpha}(F[\bar\Delta[k-1]])^{\ab}=\left\{
\begin{array}{lcl}
\Z(\tilde\alpha)&\textrm{ for }& q=k-1,\\
\Z(d_0\tilde\alpha=\alpha)&\textrm{ for }& q=k-2,\\
0&&\textrm{otherwise}\\
\end{array}\right.
$$
and $\theta_{\tilde\alpha}(\sigma_{k-1})=\tilde\alpha$.
\end{proof}

Now from the above Lemma, the simplicial monomorphism
$$
\phi_{\alpha}\colon F[S^{k-2}]\longrightarrow \AP_*
$$
is given by the composite
$$
F[S^{k-2}]\rInto^{\iota}F[\bar\Delta[k-1]]\rInto^{\theta_{\tilde\alpha}} F[S^1]\rInto^{\Theta} \AP_*.
$$
It follows that $\Theta\circ\theta_{\tilde\alpha}\colon F[\bar\Delta[k-1]]\to \AP_*$ induces a simplicial monomorphism
\begin{equation}\label{equation4.3}
F[\bar\Delta[k-1]]\ast_{F[S^{k-2}]}F[\bar\Delta[k-1]]\rInto \AP_*\ast_{F[S^{k-2}]}\AP_*,
\end{equation}
which is a homotopy equivalence by Theorem~\ref{theorem2.1}. Let $\sigma'_{k-1}$ denote the element $\sigma_{k-1}$ in second copy of $F[\bar\Delta[k-1]]$ in the free product with amalgamation $F[\bar\Delta[k-1]]\ast_{F[S^{k-2}]}F[\bar\Delta[k-1]]$. Let
$$
z_{k-1}=\sigma_{k-1}(\sigma'_{k-1})^{-1}\in \left(F[\bar\Delta[k-1]]\ast_{F[S^{k-2}]}F[\bar\Delta[k-1]]\right)_{k-1}.
$$
Then $z_{k-1}$ is a Moore cycle because
$$
d_jz_{k-1}=d_j\sigma_{k-1} (d_j\sigma'_{k-1})^{-1}=1
$$
for $j>0$ in $F[\bar\Delta[k-1]]\ast_{F[S^{k-2}]}F[\bar\Delta[k-1]]$ and
$$
d_0z_{k-1}=d_0\sigma_{k-1} (d_0\sigma'_{k-1})^{-1}=1
$$
since $d_0\sigma_{k-1}=d_0\sigma'_{k-1}$ lies in the amalgamated
subgroup $F[S^{k-2}]$. Let $f_{z_{k-1}}\colon S^{k-1}\to
F[\bar\Delta[k-1]]\ast_{F[S^{k-2}]}F[\bar\Delta[k-1]]$ be the
representing map of $z_{k-1}$ and let
$$
\tilde f_{z_{k-1}}\colon F[S^{k-1}]\longrightarrow F[\bar\Delta[k-1]]\ast_{F[S^{k-2}]}F[\bar\Delta[k-1]]
$$
be the simplicial homomorphism induced by $f_{z_{k-1}}$.

\begin{lem}\label{lemma4.2}
Let $\tilde f_{z_{k-1}}$ be defined as above.
Then
\begin{enumerate}
\item[1)] $\tilde f_{z_{k-1}}$ is a simplicial monomorphism.
\item[2)] $\tilde f_{z_{k-1}}$ is a homotopy equivalence.
\end{enumerate}
\end{lem}
\begin{proof}
(1). Observe that
$$
F[\bar\Delta[k-1]]\ast_{F[S^{k-2}]}F[\bar\Delta[k-1]]=F[\bar\Delta[k-1]\cup\bar\Delta[k-1]]
$$
is a simplicial free group, where
$\bar\Delta[k-1]\cup\bar\Delta[k-1]$ is the simplicial union by
identification $d_0\sigma_{k-1}$ with $d_0\sigma'_{k-1}$.
Assertion (1) follows from the lines of the proof of
Lemma~\ref{lemma3.2}.

(2). Since
$$
\tilde f_{z_{k-1}}\colon F[S^{k-1}]\simeq \Omega S^k\longrightarrow F[\bar\Delta[k-1]]\ast_{F[S^{k-2}]}F[\bar\Delta[k-1]]\simeq \Omega S^k
$$
is a simplicial homomorphism, it is a loop map. Thus it suffices to show that $\tilde f_{z_{k-1}}$ induces an isomorphism
$$
\tilde f_{z_{k-1}\ast}\colon\pi_{k-1}(F[S^{k-1}])\cong \Z\rTo \pi_{k-1}(F[\bar\Delta[k-1]]\ast_{F[S^{k-2}]}F[\bar\Delta[k-1]])\cong\Z.
$$
Note that
$$
\pi_{k-1}(F[\bar\Delta[k-1]]\ast_{F[S^{k-2}]}F[\bar\Delta[k-1]])\cong \pi_{k-1}(F[\bar\Delta[k-1]\cup\bar\Delta[k-1]]^{\ab}).
$$
Now the Moore chain complex of $F[\bar\Delta[k-1]\cup\bar\Delta[k-1]]^{\ab}$ is given by
$$
N_qF[\bar\Delta[k-1]\cup\bar\Delta[k-1]]^{\ab}=\left\{
\begin{array}{lcl}
\Z(\sigma_{k-1})\oplus\Z(\sigma'_{k-1})&\textrm{ for }& q=k-1,\\
\Z(d_0\sigma_{k-1}=d_0\sigma'_{k-1})&\textrm{ for }&q=k-2,\\
0&&\textrm{otherwise.}\\
\end{array}\right.
$$
Thus $\pi_{k-1}(F[\bar\Delta[k-1]\cup\bar\Delta[k-1]]^{\ab})$ is
generated by $\sigma_{k-1}-\sigma'_{k-1}$, which is the image of
$z_{k-1}$ in the abelianization
$F[\bar\Delta[k-1]\cup\bar\Delta[k-1]]^{\ab}$. It follows that
$$
\tilde f_{z_{k-1}\ast}\colon\pi_{k-1}(F[S^{k-1}])\rTo \pi_{k-1}(F[\bar\Delta[k-1]]\ast_{F[S^{k-2}]}F[\bar\Delta[k-1]]).
$$
is an isomorphism and hence the result.
\end{proof}

\subsection{Description for $\pi_*(M(\Z/q,k))$ with $k\geq3$}
With the preparation in the previous subsection, we can now
construct a simplicial group model for $\Omega M(\Z/q,k)$ with
$k\geq3$. Let $\alpha\in \calZ_{k-1}F[S^1]$ be a Moore cycle with
$\alpha\not=1$ and let $\tilde\alpha\in N_{k-1}F[S^1]$ be a Moore
chain such that $d_0\tilde\alpha\not=1$. From Lemma~\ref{lemma4.2}
together with isomorphism~(\ref{equation4.3}), there is a
simplicial monomorphism
$$
\delta_{\tilde\alpha}\colon F[S^{k-1}]\longrightarrow \calT(S^k; d_0\tilde\alpha),
$$
which is a homotopy equivalence. Let
$$
F[q]\colon F[S^{k-1}]\longrightarrow F[S^{k-1}]
$$
be the simplicial homomorphism such that
$$
F[q](x)=x^q
$$
for $x\in S^{k-1}\subseteq F[S^{k-1}]$. Clearly $F[q]$ is a
simplicial monomorphism. Now define the simplicial group
$\calT(M(\Z/q,k);\alpha)$ to be the free product with amalgamation
\begin{diagram}
F[S^{k-1}]&\rInto^{\delta_{\alpha}\circ F[q]}& \calT(S^k;d_0\tilde\alpha)\\
\dInto>{\phi_{\alpha}}&&\dTo\\
\AP_*&\rTo&\calT(M(\Z/q,k);\tilde\alpha, \alpha)=\calT(S^k;d_0\tilde\alpha)\ast_{F[S^{k-1}]}\AP_*.\\
\end{diagram}
The construction of $\delta_{\tilde\alpha}\circ F[q]$ is explicitly given as follows:
\begin{enumerate}
\item[]\textit{Regard $\tilde\alpha$ as in $k$-strand braid through the embedding $\Theta\colon F[S^1]\to \AP_*$. Let $\tilde\alpha'$ be a copy of $\tilde\alpha$ for the second copy of $\AP_*$ in the free product with amalgamation
$$
\calT(S^k;d_0\tilde\alpha)=\AP_*\ast_{F[S^{k-2}]}\AP_*.
$$
Let $\sigma_{k-1}$ be the non-degenerate element in $S^{k-1}_{k-1}$. Then
$$\delta_{\tilde\alpha}\circ F[q]\colon F[S^{k-1}]\to \calT(S^k;d_0\tilde\alpha)$$ is the unique simplicial homomorphism such that $\delta(\sigma_{k-1})=(\tilde\alpha (\tilde\alpha')^{-1})^q$. In the language of braids, $\delta_{\tilde\alpha}\circ F[q](F[S^{k-1}])$ is the subgroup of $\calT(S^k;d_0\tilde\alpha)=\AP_*\ast_{F[S^{k-2}]}\AP_*$ generated by the cablings of $(\tilde\alpha (\tilde\alpha')^{-1})^q$ in the self free product with amalgamation of braid groups.}
\end{enumerate}

One interesting point in the simplicial group
$$
\calT(M(\Z/q,k);\tilde\alpha, \alpha)=(\AP_*\ast_{F[S^{k-2}]}\AP_*)\ast_{F[S^{k-1}]}\AP_*
$$
is that we identify the $q$-th power $(\tilde\alpha (\tilde\alpha')^{-1})^q\in \AP_*\ast_{F[S^{k-2}]}\AP_*$ with $\alpha\in \AP_*$. So the cablings of $\alpha$ have $q$-th roots in $\calT(M(\Z/q,k);\tilde\alpha, \alpha)$.

\begin{thm}\label{theorem4.3}
Let $\alpha\in \calZ_{k-1}F[S^1]$ be a Moore cycle with $\alpha\not=1$ and let $\tilde\alpha\in N_{k-1}F[S^1]$ be a Moore chain such that $d_0\tilde\alpha\not=1$. Then
the simplicial group
$
\calT(M(\Z/q,k);\tilde\alpha, \alpha)
$ is homotopy equivalent to the loop space $\Omega M(\Z/q,k)$ of the Moore space.
Moreover the canonical inclusion
$$
\calT(S^k;d_0\tilde \alpha)\rInto \calT(M(\Z/q,k);\tilde\alpha, \alpha)
$$
is homotopic to the looping of the inclusion $S^k\hookrightarrow M(\Z/q,k)$.
\end{thm}
\begin{proof}
By Theorem~\ref{theorem2.1}, the classifying space $\bar W(\calT(M(\Z/q,k);\tilde\alpha, \alpha))$ is given by the homotopy push-out
\begin{diagram}
S^k&\rTo^{\bar W(\delta_{\tilde\alpha}\circ F[q])}& S^k\simeq \bar W(S^k;d_0\tilde\alpha)\\
\dTo&&\dTo\\
\ast&\rTo& \bar W(\calT(M(\Z/q,k);\tilde\alpha, \alpha)).\\
\end{diagram}
Since
$$
\bar W(\delta_{\tilde\alpha}\circ F[q]_*)\colon
\pi_k(S^k)\cong\pi_{k-1}(F[S^{k-1}])\longrightarrow
\pi_k(S^k)\cong \pi_{k-1}(\calT(S^k;d_0\tilde\alpha))
$$
is of degree $q$, $\bar W(\calT(M(\Z/q,k);\tilde\alpha, \alpha))\simeq M(\Z/q,k)$. Observe that the right column of above diagram is homotopic to the inclusion of the bottom cell $S^k\hookrightarrow M(\Z/q,k)$. The assertions follow.
\end{proof}

Let $A_{i,j}, A'_{i,j}$ and $A''_{i,j}$ be copies of $A_{i,j}$ for generators for $P_n$ in the free product with amalgamation
$$
\calT(M(\Z/q,k);\tilde\alpha,\alpha)_{n-1}=(P_n\ast_{F[S^{k-2}]_{n-1}}P_n)\ast_{F[S^{k-1}]_{n-1}}P_n
$$
and let $R_{i,j}$ be the normal closure of $A_{i,j},A'_{i,j}$ and $A''_{i,j}$ in $\calT(M(\Z/q,k);\tilde\alpha,\alpha)_{n-1}$.

\begin{thm}\label{theorem4.4}
Let $k\geq 3$. Let $\alpha\in \calZ_{k-1}F[S^1]$ be a Moore cycle with $\alpha\not=1$ and let $\tilde\alpha\in N_{k-1}F[S^1]$ be a Moore chain such that $d_0\tilde\alpha\not=1$. Then $\pi_n(M(\Z/q,k))$ is isomorphic to the center of the group
$$
((P_n\ast_{F[S^{k-2}]_{n-1}}P_n)\ast_{F[S^{k-1}]_{n-1}}P_n)/[R_{i,j}\ | \ 1\leq i<j\leq n]_S
$$
for any $n$.
\end{thm}
\begin{proof}
Since $\calT(M(\Z/q,k);\tilde\alpha,\alpha)$ is a simplicial quotient group of $\calG\ast\calG\ast\calG$, the Moore boundaries
$$
\calB_{n-1}\calT(M(\Z/q,k);\tilde\alpha,\alpha)=[R_{i,j}\ | \ 1\leq i<j\leq n]_S
$$
by Lemma~\ref{lemma3.5}. Observe that the group
$(P_m\ast_{F[S^{k-2}]_{m-1}}P_n)\ast_{F[S^{k-1}]_{m-1}}P_m$ has
trivial center by Lemma \ref{cent}. The assertion follows
from~\cite[Proposition 2.14]{Wu2}.
\end{proof}

\begin{rem}
{\rm An explicit choice of $\alpha$ and $\tilde\alpha$ can be given. For instance, we can choose
$$
\alpha_{k+1}=[[[x_1^{-1}, x_1x_2^{-1}],x_2x_3^{-1}],\ldots, x_{k-2}x_{k-1}^{-1},x_{k-1}]
$$
in Theorem~\ref{theorem1.2} as a $k$-strand Brunnian braid and choose
$$
\tilde\alpha_k=[[ x_1x_2^{-1},x_2x_3^{-1}],\ldots, x_{k-2}x_{k-1}^{-1},x_{k-1}]
$$
as a $k$-strand quasi-Brunnian braid in the sense of~\cite{CW1}. Then we obtain an explicit simplicial group model $\calT(M(\Z/q,k);\tilde\alpha_k,\alpha_{k+1})$ for $\Omega M(\Z/q,k)$.}\hfill $\Box$
\end{rem}

\vspace{.5cm}
\section{Description of the homotopy groups of Moore Spaces $M(\Z/q,2)$ and Proof of Theorem~\ref{theorem1.3}}\label{section5}
\vspace{.5cm} Let $\calT(M(\Z/q,2))$ be the free product with
amalgamation by the following diagram
\begin{diagram}
F[S^1]&\rInto^{F[q]}&F[S^1]\\
\dInto>{\Theta}&&\dTo\\
\AP_*&\rTo&\calT(M(\Z/q,2))=\AP_*\ast_{F[S^1]}F[S^1].\\
\end{diagram}
By Theorem~\ref{theorem2.1}, there is a homotopy push-out
\begin{diagram}
S^2\simeq \bar W F[S^1]&\rInto^{\bar F[q]\simeq [q]}&S^2\simeq \bar WF[S^1]\\
\dInto>{\Theta}&&\dTo\\
\bar W \AP_*\simeq \ast&\rTo&\bar W\calT(M(\Z/q,2))\\
\end{diagram}
and so $\bar W\calT(M(\Z/q,2)\simeq M(\Z/q,2)$. Namely
$\calT(M(\Z/q,2))$ is a simplicial group model for $\Omega
M(\Z/q,2)$.

For each $n$, the homomorphism
$$
F[q]\colon F[S^1]_{n-1}=F_{n-1}\longrightarrow F[S^1]_{n-1}=F_{n-1}
$$
is the homomorphism $\phi_q$ described in Theorem~\ref{theorem1.3}. Thus as a group
$$
\calT(M(\Z/q,2))_{n-1}=P_n\ast_{\phi_q}F_{n-1}.
$$
We give an more explicit description of the group $\calT(M(\Z/q,2))_{n-1}$ using degeneracy operations. Let $\{x_j\}_{j=1,\dots, n-1}$ be the set of generators for
$F_{n-1}=F[S^1]_{n-1}$ as the second factor in the free product
$P_n\ast_{\phi_q}F_{n-1}$ for $1\leq j\leq n-1$. (\textbf{Note.} In the introduction to Theorem~\ref{theorem1.3}, we write $y_j$ for $x_j$.) As an element in $F[S^1]_{n-1}$,
$$
x_j=s_{n-2}\cdots s_{j+1}s_j s_{j-2}s_{j-3}\cdots s_1s_0\sigma_1
$$
for $1\leq j\leq n-1$. The group $\calT(M(\Z/q,2))_{n-1}$ is the quotient group of $P_n\ast F_{n-1}$ by the relation
$$
s_{j+1}s_j s_{j-2}s_{j-3}\cdots s_1s_0A_{1,2}=x_j^q
$$
for $1\leq j\leq n-1$, where $s_{j+1}s_j s_{j-2}s_{j-3}\cdots s_1s_0A_{1,2}$ the cabling of $A_{1,2}$ as the picture in the introduction.

Let $z_1=x_1,\
z_n=x_{n-1}$ and $z_i=x_ix_{i-1}^{-1},$ for $i=2,\dots, n-1$. Now
let $R_i=\la z_i\ra^{P_n\ast_{\phi_q}F_{n-1}}$ be the normal
closure of $z_i$ in $P_n\ast_{\phi_q}F_{n-1}$ for $1\leq i\leq n$
and let $R_{s,t}=\la A_{s,t}\ra^{P_n\ast_{\phi_q}F_{n-1}}$ be the
normal closure of $A_{s,t}$ in $P_n\ast_{\phi_q}F_{n-1}$ for
$1\leq s<t\leq n$. Define the index set $\Index(R_j)=\{j\}$ for
$1\leq j\leq n$ and $\Index(R_{s,t})=\{s,t\}$ for $1\leq s<t\leq
n$. Now define the symmetric commutator subgroup
$$
[R_i, R_{s,t}\ | \ 1\leq i\leq n,1\leq s<t\leq
n]_S=\prod_{\{1,2,\ldots,n\}=\bigcup\limits_{j=1}^t\Index(C_j)}[[C_1,C_2],\ldots,C_t],
$$
where each $C_j=R_i$ or $R_{s,t}$ for some $i$ or $(s,t)$.

\begin{thm}[Theorem~\ref{theorem1.3}]
The homotopy group $\pi_n(M(\Z/q,2))$ is isomorphic to the center
of the group
$$
(P_n\ast_{\phi_q}F_{n-1})/[R_i, R_{s,t}\ | \ 1\leq i\leq n,1\leq
s<t\leq n]_S
$$
for any $n>3$.
\end{thm}
\begin{proof}
The proof is similar to that of Theorem~\ref{theorem1.2}. It is easy to see that the group $\calT(M(\Z/q,2))_{m}=P_{m+1}\ast_{\phi_q}F_{m}$ has the trivial center for $m\geq 2$. From~\cite[Proposition 2.14]{Wu2}, $\pi_m(\calT(M(\Z/q,2)))\cong\pi_{m+1}(M(\Z/q,2))$ is isomorphic to the center of $\calT(M(\Z/2))_m/\calB_m\calT(M(\Z/q,2))$ for $m\geq 3$. Thus the key point is to show the Moore boundaries
$$
\calB_{n-1}\calT(M(\Z/q,2))=[R_i, R_{s,t}\ | \ 1\leq i\leq n,1\leq
s<t\leq n]_S.
$$
We construct a simplicial group $\tilde F$ by $\tilde F_{n-1}$ generated by the letters $z_1,\ldots,z_n$ with face operation
$$
d_jz_k=\left\{
\begin{array}{lcl}
z_k&\textrm{ for }& k<j+1\\
1&\textrm{ for }&k=j+1\\
z_{k-1}&\textrm{ for }& k>j+1\\
\end{array}
\right.
$$
and degeneracy operations
$$
s_jz_k=\left\{
\begin{array}{lcl}
z_k&\textrm{ for }& k<j+1\\
z_{j+1}z_{j+2}&\textrm{ for }&k=j+1\\
z_{k+1}&\textrm{ for }& k>j+1\\
\end{array}
\right.
$$
for $0\leq j\leq n-1$. Then $\tilde F$ is a simplicial group with a simplicial epimorphism $f\colon \tilde F\to F[S^1]$ by sending the letter $z_j$ of $\tilde F_{n-1}$ to the element $z_j\in F[S^1]_{n-1}$. Let $g\colon \calG\to \AP_*$ be the canonical simplicial epimorphism. Then we have the simplicial epimorphism
$$
\calG\ast\tilde F\rOnto \AP_*\ast F[S^1]\rOnto \calT(M(\Z/q,2)).
$$
Observe that $\Ker(d_n\colon (\calG\ast\tilde F)_n\to
(\calG\ast\tilde F)_{n-1})$ is the normal closure of the elements
$x_{i,n+1}, z_{n+1}$. By repeating the arguments in the proof of
Lemma~\ref{lemma3.5}, we have
$$
\calB_{n-1}(\calG\ast \tilde F)=[R_i,R_{s,t}\ | \ 1\leq i\leq n, \ 1\leq s<t\leq n]_S
$$
and hence the result.
\end{proof}

\noindent{\bf Example.} Consider the case $n=3$. The group
$$G=(P_3\ast_{\phi_q}F_{2})/[R_i, R_{s,t}\ | \ 1\leq i\leq 3,1\leq
s<t\leq 3]_S$$ is given by generators $x_1,x_2,
a_{12},a_{13},a_{23}$ and the following relations
\begin{align*}
& x_1^q=a_{12}a_{13},\ x_2^q=a_{13}a_{23},\\
& [[x_1^{g_1},x_2^{g_2}],x_1]=[[x_1^{g_1},x_2^{g_2}],x_2]=1,\ g_1,g_2\in G\\
& [a_{12}^g,a_{13}]=[a_{12}^g,a_{23}]=[a_{13}^g,a_{23}]=1,\ g\in
G\\
& [x_1^g, a_{23}]=[(x_1x_2^{-1})^g,a_{13}]=[x_2^g,a_{12}]=1,\ g\in
G.
\end{align*}
Presenting $a_{13}, a_{23}$ via generators $x_1,x_2,a_{12}$, we
get the following 3-generator presentation of $G$:
\begin{align*}
& [[x_1^{g_1},x_2^{g_2}],x_1]=[[x_1^{g_1},x_2^{g_2}],x_2]=1,\ g_1,g_2\in G\\
&
[a_{12}^g,a_{12}^{-1}x_1^q]=[a_{12}^g,x_1^{-q}a_{12}x_2^q]=[(a_{12}^{-1}x_1^q)^g,x_1^{-q}a_{12}x_2^q]=1,\
g\in
G\\
& [x_1^g,
x_1^{-q}a_{12}x_2^q]=[(x_1x_2^{-1})^g,a_{12}^{-1}x_1^q]=[x_2^g,a_{12}]=1,\
g\in G.
\end{align*}
Straightforward computations show that $G$ is a 3-generator
nilpotent group of class 2, given by generators $x_1,x_2,a_{12}$
and relations
\begin{align*}
&
[a_{12},x_2]=[a_{12},x_1^q]=[x_1^q,x_2^q]=[x_1,a_{12}x_2^q]=[x_1x_2^{-1},a_{12}^{-1}x_1^q]=1\\
& [[G,G],G]=1
\end{align*}
It follows that the order of the element $[x_1,x_2]$ is $(2q,q^2)$
in $G$. The center of $G$ is bigger than the subgroup generated by
$[x_1,x_2]$, since $a_{12}^q$ lies in the center. Denote
$Z_1=\langle a_{12},x_1\rangle^G,\ Z_2=\langle
a_{12},x_1x_2^{-1}\rangle^G,\ Z_3=\langle x_2\rangle^G.$ The
homotopy group $\pi_3M(Z/q,2)$ is given now as the intersection
$$
Z_1\cap Z_2\cap Z_3\simeq \mathbb Z/(2q,q^2).
$$

\end{document}